%% file: basechange.tex
\documentclass{amsart}
\chardef\bslash=`\\

\usepackage{amssymb,mathdots,mathrsfs,amsmath,amsfonts,amsthm,amscd,stmaryrd,enumitem,color,microtype}
\usepackage[all,cmtip,poly]{xy}
\usepackage[hidelinks]{hyperref}

\newtheorem{theorem}[subsection]{Theorem}
\newtheorem{corollary}[subsection]{Corollary}

\newtheorem{lemma}[subsection]{Lemma}
\newtheorem{lem}[subsection]{Lemma}

\newtheorem{cor}[subsection]{Corollary}

\newtheorem{prop}[subsection]{Proposition}
\newtheorem{proposition}[subsection]{Proposition}

\newtheorem{defn}[subsection]{Definition}
\theoremstyle{remark}
\newtheorem{remark}[subsection]{Remark}
\newtheorem{rem}[subsection]{Remark}

\numberwithin{equation}{section}

\newif\iffinalrun
\iffinalrun
\else
 \fi

\iffinalrun
  \newcommand{\need}[1]{}
  \newcommand{\mar}[1]{}
\else
  \newcommand{\need}[1]{{\tiny *** #1}}
  \newcommand{\mar}[1]{\marginpar{\raggedright\tiny FIXME:  #1}}\fi

\renewcommand\mathbb{\mathbf}

\newcommand{\gog}{{\mathfrak{g}}}

\newcommand{\wt}{{\operatorname{wt}}}

\newcommand{\rec}{{\operatorname{rec}}}

\newcommand{\rbar}{\overline{r}}

\newcommand{\wv}{{\widetilde{v}}}
\newcommand{\ww}{{\widetilde{w}}}

\renewcommand{\ell}{l}

\def\PSL{\mathrm{PSL}}
\def\PGL{\mathrm{PGL}}

\newcommand{\ad}{\operatorname{ad}}

\newcommand{\tr}{\operatorname{tr}}

\newcommand{\A}{\mathbf{A}}
\newcommand{\bA}{\ensuremath{\mathbf{A}}}

\newcommand{\bC}{\ensuremath{\mathbf{C}}}
\newcommand{\Cp}{\ensuremath{\mathbf{C}_p}}

\newcommand{\bF}{\ensuremath{\mathbf{F}}}
\newcommand{\bG}{\ensuremath{\mathbf{G}}}

\newcommand{\bQ}{\ensuremath{\mathbf{Q}}}
\newcommand{\Q}{\QQ}
\newcommand{\QQ}{{\mathbb Q}}

\newcommand{\bT}{\ensuremath{\mathbf{T}}}

\newcommand{\Z}{\ZZ}
\newcommand{\ZZ}{{\mathbb Z}}
\newcommand{\bZ}{\ensuremath{\mathbf{Z}}}

\newcommand{\cA}{{\mathcal A}}
\newcommand{\cC}{{\mathcal C}}
\newcommand{\cD}{{\mathcal D}}
\newcommand{\cE}{{\mathcal E}}
\newcommand{\cF}{{\mathcal F}}

\newcommand{\cL}{{\mathcal L}}

\newcommand{\cO}{{\mathcal O}}
\newcommand{\cR}{{\mathcal R}}
\newcommand{\cS}{{\mathcal S}}
\newcommand{\cT}{{\mathcal T}}

\newcommand{\cW}{{\mathcal W}}

\newcommand{\cZ}{{\mathcal{Z}}}

\newcommand{\cX}{{\mathcal{X}}}

\newcommand{\Qbar}{\overline{\Q}}
\newcommand{\Zbar}{\overline{\Z}}

\newcommand{\Zpbar}{\Zbar_p}

\newcommand{\Zpbarx}{\Zpbar^{\times}}

\newcommand{\Qp}{\Q_p}

\newcommand{\Qpbar}{\Qbar_p}

\DeclareMathOperator{\End}{End}

\DeclareMathOperator{\Fil}{Fil}

\DeclareMathOperator{\Gal}{Gal}
\newcommand{\GL}{\mathrm{GL}}

\DeclareMathOperator{\Hom}{Hom}

\DeclareMathOperator{\Ind}{Ind}

\DeclareMathOperator{\SL}{SL}

\DeclareMathOperator{\Spec}{Spec}

\DeclareMathOperator{\Sym}{Symm}
\DeclareMathOperator{\Symm}{Symm}

\newcommand{\Frob}{\mathrm{Frob}}

\newcommand{\toisom}{\buildrel\sim\over\to}
\newcommand{\Art}{{\operatorname{Art}}}
\newcommand{\Res}{\operatorname{Res}}

\newcommand{\univ}{\mathrm{univ}}

\newcommand{\doubleslash}{/\kern-0.2em{/}}

\begin{document}
\author{Laurent Clozel,  James Newton, and Jack  A. Thorne}
\title[Non-abelian base change]{Non-abelian base change for symmetric power liftings of  holomorphic modular forms}
\begin{abstract}
Let $f$ be a non-CM Hecke eigenform of weight $k \geq 2$. We give a new proof of some cases of Langlands functoriality for the automorphic representation $\pi$ associated to $f$. More precisely, we prove the existence of the base change lifting, with respect to any totally real extension $F / \bQ$, of any symmetric power lifting of $\pi$. 
\end{abstract}
\maketitle
\setcounter{tocdepth}{1}
\tableofcontents

\section{Introduction}
	
	Let $F$ be a totally real number field. In this paper we prove the following theorem:
	\begin{theorem}\label{intro_thm_existence_of_base_change}
Let $\pi$ be a cuspidal, regular algebraic automorphic representation of $\GL_2(\bA_\bQ)$ without CM, and let $n \geq 1$. Then the base change of the symmetric power $\Sym^{n-1} \pi$ to $F$ exists. More precisely, there exists a cuspidal, regular algebraic automorphic representation $\mathrm{BC}_{F / \bQ}(\Sym^{n-1} \pi)$ of $\GL_n(\bA_F)$ such that for any place $w$ of $F$ lying above the place $v$ of $\bQ$, we have
\[ \rec_{F_w}(\mathrm{BC}_{F / \bQ}(\Sym^{n-1} \pi)_w) \cong \Symm^{n-1}\rec_{\bQ_{v}}(\pi_{v})|_{W_{F_w}}. \]
	\end{theorem}
More informally, we prove the existence of certain functorial liftings for Hecke eigenforms of weight $k \geq 2$ over $\bQ$; the functoriality being the composite of `$(n-1)$th symmetric power' and `base change with respect to $F / \bQ$'. 

The case $n = 2$ of Theorem \ref{intro_thm_existence_of_base_change}, which is just base change, was established by Dieulefait using modularity lifting theorems and the notion of a `safe chain' of congruences between modular forms \cite{Dieulefait-bc,Die15}. The general case of Theorem \ref{intro_thm_existence_of_base_change} then follows using the results of two of us on symmetric power functoriality for Hilbert modular forms \cite{NT2022symmetric}. However, in this paper we give new proofs of both base change and the existence of the symmetric power lifting for the base changed eigenform. Our strategy is based on the papers \cite{New21a, New21b}, which handle the case $F = \bQ$. We hope that this work (a first draft of which was written before \cite{NT2022symmetric}) is still of interest.

The case $F = \bQ$ was established in the papers \cite{New21a, New21b} using the following two principles as the primary inputs:
\begin{itemize}
\item An `analytic continuation of functoriality' principle, that asserts that the existence of a functorial lifting can be propagated along irreducible components of the eigencurve. This is particularly effective in light of the description by Buzzard--Kilford \cite{Buz05} of a large part of the 2-adic, tame level 1 eigencurve. 
\item A `functoriality lifting theorem', giving robust hypotheses under which a congruence $\overline{r}_{\pi, \iota} \cong \overline{r}_{\pi', \iota}$ of mod $p$ residual representations attached to (cuspidal, regular algebraic) automorphic representations of $\GL_2(\bA_F)$ implies the equivalence of the existence of the liftings $\Sym^{n-1} \pi$ and $\Sym^{n-1} \pi'$.
\end{itemize}
Here we extend these ideas in order to prove Theorem \ref{intro_thm_existence_of_base_change}. Where possible, we strengthen the arguments used in \cite{New21a, New21b}, with one eye towards potential future applications of a similar type. In particular, we prove a more general classicality result (Proposition \ref{prop:noncritical implies classical}), and replace the appeal to the very intricate level-raising results of \cite{New21a} with an appeal to the more generally applicable level-raising theorems proved in \cite{Tho22}. This gives a streamlined proof even in the case $F= \Q$. 

\subsection*{Notation}
We follow the notation and conventions established in the notation section of \cite[pp.~6-9]{New21a}.  The following table gives a list of symbols used with their meanings. We refer to \emph{loc.~cit.} for precise definitions.

	\begin{tabular}{|p{0.2\textwidth}|p{0.7\textwidth}|}
		\hline
		Symbol &  Meaning \\ \hline
		$G_F$ & Absolute Galois group of a field $F$ of characteristic 0 \\ \hline
		$F_v$, $\cO_{F_v}$, $\varpi_v$ &  Completion of number field $F$ at finite place $v$, ring of integers, fixed choice of uniformizer \\  \hline
		$F_S / F$, $G_{F, S}$& Maximal extension of number field $F$ unramified outside finite set $S$, $\Gal(F_S / F)$ \\ \hline
		$E, \cO, \varpi, k$ &  Finite extension of $\bQ_p$ with ring of integers, uniformizer, residue field \\ \hline
		$\cC_\cO$ & Category of complete Noetherian local $\cO$-algebras with residue field $k$ \\ \hline
		$W_K, I_K, \Art_K$ & Weil group, inertia group, Artin map of $p$-adic local field $K$ \\ \hline
		$\rec_K$ & Local Langlands correspondence for $\GL_n(K)$  \\ \hline
		$r_{\pi, \iota}$ & $p$-adic Galois representation associated to a regular algebraic, cuspidal, polarizable automorphic representation of $\GL_n(\A_F)$, $F$ a CM or totally real number field, and $\iota : \overline{\bQ}_p \to \bC$ an isomorphism \\ \hline
	\end{tabular}
\vskip 10pt
	\subsection*{Acknowledgements}

J.N. was supported by a UKRI Future Leaders Fellowship, grant MR/V021931/1. J.T.’s work received funding from the European Research Council (ERC) under the European Union’s Horizon 2020 research and innovation programme (grant agreement No 714405). For the purpose of Open Access, the authors have applied a CC BY public copyright licence to any Author Accepted Manuscript (AAM) version arising from this submission.

\section{Preliminaries}\label{sec_preliminaries}

We first remark  that  if $\pi$ is a cuspidal, regular algebraic automorphic representation of $\GL_2(\bA_\bQ)$ without CM, $n \geq 2$ is an integer, $p$ is a prime, and $\iota : \overline{\bQ}_p \to \bC$ is an isomorphism, then $\Sym^{n-1} r_{\pi, \iota}|_{G_F}$ is irreducible. This explains why the base change of $\Sym^{n-1} \pi$ to a totally real field $F$ can always be expected to be \emph{cuspidal}. Indeed, if $\pi$ is not of CM type, then the Zariski closure of $r_{\pi,  \iota}(G_{\bQ})$ contains $\SL_2$, implying that in fact any finite index  subgroup of $G_\bQ$ acts irreducibly in any symmetric power of the standard representation. This is quite useful: 
\begin{lemma}\label{lem_OK_after_base_change}
Let $\pi$ be a cuspidal, regular algebraic automorphic representation of $\GL_2(\bA_\bQ)$, let $n \geq 1$, and let $F' / F$ be a (Galois) soluble extension of totally real number fields. Then the following are equivalent:
\begin{enumerate}
\item There exists a polarizable
automorphic representation $\Pi'$ of $\GL_n(\bA_{F'})$, a prime number $p$, and an isomorphism $\iota : \overline{\bQ}_p \to \bC$ such that $r_{\Pi', \iota} \cong \Sym^{n-1} r_{\pi, \iota}|_{G_{F'}}$.
\item There exists a polarizable automorphic representation $\Pi$ of $\GL_n(\bA_F)$ such that for any prime number $p$ and isomorphism $\iota : \overline{\bQ}_p \to \bC$, we have $r_{\Pi, \iota} \cong \Sym^{n-1} r_{\pi, \iota}|_{G_F}$.
\end{enumerate}
\end{lemma}
\begin{proof}
This follows from the irreducibility of $\Sym^{n-1} r_{\pi, \iota}|_{G_{F'}}$ and soluble base change for $\GL_n$ (see \cite{Art89}; a convenient reference for us is \cite[Lemma 1.3]{blght}).
\end{proof}
\begin{lemma}
Let $F$ be a totally real field, and let $\pi$ be cuspidal, regular algebraic automorphic representation of $\GL_2(\bA_F)$ without CM. Let $n \geq 1$, let $p$ be a prime number, and let $\iota : \overline{\bQ}_p \to \bC$ be an isomorphism. Then the following are equivalent:
\begin{enumerate}
\item $\Symm^{n-1} \pi$ exists, in the sense that there is a cuspidal automorphic representation $\Pi = \Symm^{n-1} \pi$ of $\GL_n(\bA_F)$ such that for every place $v$ of $F$, $\rec_{F_v}(\Pi_v) \cong \Symm^{n-1} \, \rec_{F_v} \pi_v$. $\Pi$ is then polarizable.
\item There exists a polarizable automorphic representation $\Pi$ of $\GL_n(\bA_F)$ such that $r_{\Pi, \iota} \cong \Symm^{n-1}  r_{\pi, \iota}$.
\end{enumerate}
\end{lemma}
\begin{proof}
If (1) holds then (2) holds, as can be checked at finite unramified places. If (2) holds, then we must show that $\Pi = \Symm^{n-1} \pi$, i.e.\ that $\Pi_v$ has the correct local components. If $v$ is a finite place then this follows from local-global compatibility for $r_{\pi, \iota}$. If $v$ is an infinite place, then we have $\rec_{F_v}(\Pi_v)|_{\bC^\times} \cong \Symm^{n-1} \rec_{F_v} \pi_v|_{\bC^\times}$ (as both sides are a sum of characters, determined by the Hodge--Tate weights of $r_{\Pi, \iota} \cong \Symm^{n-1}  r_{\pi, \iota}$). We can now repeat verbatim the argument of \cite[Proof of Theorem 8.1]{Tho22} to conclude that $\rec_{F_v}(\Pi_v) \cong \Symm^{n-1} \, \rec_{F_v} \pi_v$.
\end{proof}

We next give a useful lemma, which implies that many congruences of residual Galois representations over $\bQ$ remain robust after base extension to a totally real field $F$.
\begin{lemma}\label{lem_large_image_over_F}
Let $p$ be an odd prime, and let $\overline{\rho} : G_\bQ \to \GL_2(\overline{\bF}_p)$ be a continuous representation such that $\overline{\rho}(G_\bQ)$ contains a conjugate of $\SL_2(\bF_{p^a})$ for some $a \geq 1$ such that $p^a \geq 5$. Let $F$ be a totally real number field. Then $\overline{\rho}(G_F)$ contains a conjugate of $\SL_2(\bF_{p^a})$. 
\end{lemma}
\begin{proof}
After increasing $a$ and replacing $\overline{\rho}$ by a conjugate, and applying the classification of finite subgroups of $\GL_2(\overline{\bF}_p)$, we can assume that $\operatorname{Proj} \overline{\rho}(G_\bQ)$ equals $\PSL_2(\bF_{p^a})$ or $\PGL_2(\bF_{p^a})$. In particular, it contains the simple group $\PSL_2(\bF_{p^a})$ as a subgroup of index at most 2. If $\operatorname{Proj} \overline{\rho}(G_F)$ contains $\PSL_2(\bF_{p^a})$, then we're done. 

We are free to enlarge $F$, so we can replace it by its Galois closure.  Let $K / \bQ$ be the extension cut out by $\operatorname{Proj} \overline{\rho}$. Then $\Gal(K / K \cap F)$ is a normal subgroup of $\Gal(K / \bQ)$ which contains each complex conjugation. Since $\operatorname{Proj} \overline{\rho}(G_F) = \operatorname{Proj} \overline{\rho}(G_{K \cap F})$, we see that if the conclusion of the lemma fails, then $\Gal(K / K \cap F)$ must be trivial. This would contradict the fact that since $p$ is odd, each complex conjugation is non-trivial as an element of $\Gal(K / \bQ)$. This concludes the proof. 
\end{proof}

\section{Analytic continuation of base change}

\subsection{Definite unitary groups}\label{subsec:unitary}
As in \cite{New21a}, our main analytic continuation result will first be proved for definite unitary groups. We adopt the following assumptions and notation, as in \cite[\S1]{New21a}:
\begin{itemize}
	\item $K$ is a CM number field  such that $K / K^+$ is everywhere 
	unramified (this implies that $[K^+ : \Q]$ is even).
	\item $p$ is a prime. We write $S_p$ for the set of $p$-adic places of 
	$K^+$. 
	\item $\iota$ is an isomorphism $\iota:\Qpbar\toisom\bC$.
	\item $S$ is a finite set of finite places of $K^+$, all of which split in 
	$K$. $S$ contains $S_p$.
	\item For each $v \in S$, we suppose fixed a factorization $v = \wv \wv^c$ in 
	$K$, and write $\widetilde{S} = \{ \wv \mid v \in S \}$. 
\end{itemize}

Let $n \geq 1$ be an integer. Under the above assumptions we can fix the following data:
\begin{itemize}
	\item The unitary group $G_n = G$ over $K^+$ with $R$-points given by the formula
	\begin{equation}\label{eqn_definite_unitary_group} G(R) = \{ g \in \GL_n(R \otimes_{K^+} K) \mid g =  (1 \otimes c)(g)^{-t} \}. \end{equation}
	$G$ is quasi-split at all finite places, while for each place $v | \infty$ of $K^+$, $G(K^+_v)$ is compact. We use the same formula to extend $G$ to a reductive group scheme over $\cO_{K^+}$.
\end{itemize}
\begin{itemize}	
	\item If $L^+/K^+$ is a finite totally real extension, we set $L = L^+K$ and let $G_{n,L^+}$ be the extension of scalars of $G_n$ to $L^+$ (it can be identified with the unitary group defined relative to $L/L^+$ by formula \ref{eqn_definite_unitary_group}).
	\item If $T$ is a set of places of $K^+$ then we write $T_L$ for the set of 
	places of $L^+$ lying above $T$. If $w \in T_L$ lies above $v \in T$ and $v$ 
	splits $v = \wv \wv^c$ in $K$ (in particular, we suppose that we have made a choice of $\wv|v$), then we will write $\ww$ for the unique place of 
	$L$ which lies above both $w$ and $\wv$ (in which case $w$ splits $w = 
	\ww \ww^c$ in $L$). We write e.g.\ $\widetilde{S}_L$ for the set of places of the form $\ww$ ($w \in S_L$).
\end{itemize}

\subsection{Eigenvarieties, their classical points and Galois pseudocharacters}
Let $E \subset \Qpbar$ be a coefficient field, which contains the normal closure of $L$, and denote its residue field by $k$. Let $U_{n,L} = \prod_v U_{n,L,v} \subset G_{n,L^+}(\bA_{L^+}^\infty)$ be an open compact subgroup with $U_{n,L,v} = G_{n}(\cO_{L^+_v})$ for finite places $v\not\in S$. (When $L = K$ we will suppress all the subscripts in our notation.)

For $\pi$ an automorphic representation of $G_n(\bA_{L^+})$ with $\pi^{U_{n,L}}\neq 0$, we have the notion of accessible refinements $\chi = (\chi_v)_{v\in S_{p,L}}$ for $\pi$, where each $\chi_v$ is a $p$-adic character of the diagonal maximal torus $T_n(L_{\wv})$ in $G_n(L^+_v) \cong \GL_n(L_{\wv})$ which appears in the normalized Jacquet module of $\iota^{-1}\pi_v$. See \cite[\S2.17]{New21a} for more details.

We denote by $\cR\cA_{n,L}$ the set of pairs $(\pi,\chi)$, where $\pi$ is an automorphic representation of $G_n(\bA_{L^+})$ such that $(\pi^{\infty})^{U_{n,L}^p} \ne 0$ and $\chi$ is an accessible refinement of $\pi$.
	
We have an eigenvariety $\cE_{n,L}$ of tame level $U_{n,L}^p$, a rigid space defined over $E$. See \cite[\S2.18.1]{New21a}. Its (Zariski-dense) subset $Z_{n,L}$ of classical points admits a surjective map $\gamma_n: \cR\cA_{n,L} \to Z_{n,L}$. 

The eigenvariety comes equipped with a finite morphism $\nu: \cE_{n,L}\to \cT_{n,L}$, where $\cT_{n,L} = \Hom\left(\prod_{v\in S_{p,L}}T_n(L_{\wv}),\bG_m\right)$ and $\nu(\gamma_n(\pi,\chi)) = \nu(\pi,\chi)$ is given by the formula \cite[(2.18.1)]{New21a}. It is also natural to consider a renormalisation $\jmath\nu$ given by composing $\nu$ with an automorphism $\jmath$ of $\cT_{n,L}$ (see \cite[\S2.17]{New21a} again). When studying the local geometry of eigenvarieties, we will need to restrict to points $z$ where the character $\delta(z) := \jmath\nu(z)$ is regular:
\begin{defn}\label{def:regular}
	A character \[\delta = \prod_{v \in S_{p,L}}\delta_v \in \Hom\left(\prod_{v\in S_{p,L}}T_n(L_{\wv}),\Qpbar^\times\right)\] is \emph{regular} if for each $v \in S_{p,L}$ and $1\le i< j \le n$ we have $\delta_{v,i}/\delta_{v,j} \ne x^{a_v}$ for any $a_v = (a_{v,\tau})_{\tau}\in\Z_{\ge 0}^{\Hom_{\Qp}(L_{\wv},\Qpbar)}$, where $x^{a_v}(y) = \prod_{\tau}\tau(y)^{a_{v,\tau}}$.
\end{defn}
Regularity defines a Zariski-open condition on characters, so we get a Zariski-open subspace $\cT_{n,L}^{reg} \subset \cT_{n,L}$.

For a single place $v$, we also have $\cT_{n,L_{\wv}}^{reg} \subset \cT_{n,L_{\wv}}$ where $\cT_{n,L_{\wv}} = \Hom\left(T_n(L_{\wv}),\bG_m\right)$ and $\cT_{n,L_{\wv}}^{reg}$ is defined by the condition in Definition \ref{def:regular} at the place $v$.

There is a family of conjugate self-dual Galois pseudocharacters $\tau^{\univ}_{n,L} : G_{L,S_L}\to \cO(\cE_{n,L})$ over $\cE_{n,L}$, with the property that the specialization $\tau_{z}$ at $z = \gamma(\pi,\chi) \in Z_{n,L}$ is equal to $\tr r_{\pi,\iota}$. We fix a residual pseudocharacter $\overline{\tau}:G_{L,S_L}\to k$ and consider the (open and closed) subspace $\cE_{n,L}(\overline{\tau}) \subset \cE_{n,L}$ where $\tau^{\univ}_{n,L}$ reduces to $\overline{\tau}$. We assume that this subspace is non-empty (in particular, $\overline{\tau}$ is then conjugate self-dual). Then we have the space $\cX_{ps,n,L}$ of conjugate self-dual deformations of $\overline{\tau}$, defined in \cite[\S2.11.1]{New21a}.

Now we have a morphism $\lambda: \cE_{n,L}(\overline{\tau}) \to \cX_{ps,n,L}$ and a closed immersion $i_n = \lambda \times (\jmath\circ\nu): \cE_{n,L}(\overline{\tau}) \to \cX_{ps,n,L} \times \cT_{n,L}$. 

The locus $\cE_{n,L}^{\mathrm{irred}}$ where $\tau^{\univ}_{n,L}$ is absolutely irreducible is an open subset of $\cE_{n,L}$. If $z \in \cE_{n,L}^{\mathrm{irred}}(\Qpbar)$, we have an absolutely irreducible representation $r_z : G_{L,S_L} \to \GL_n(\Qpbar)$, well-defined up to conjugation, with $\tr r_z = \tau_z$. 

If $\delta_v \in \Hom\left(T_n(L_{\wv}),\Qpbar^\times\right)$ and $\rho_v: G_{L_{\wv}}\to \GL_n(\Qpbar)$ is a continuous representation, we have the notion of a triangulation of $\rho_v$ with parameter $\delta_v$ (cf.~\cite[\S2.3.1]{New21a}). When the labelled weights of the triangulation are strictly increasing sequences of integers, we say that the triangulation is non-critical. When the parameter $\delta_v$ satisfies $\delta_{v,i}(\varpi_{\wv}) \in \Zpbarx$ for each $i$ we say that it is ordinary.

\begin{defn}
	Suppose $(\pi,\chi) \in \cR\cA_{n,L}$ and let $\delta = \jmath\nu(\pi,\chi)$. The refinement $\chi$ is said to be \emph{ordinary} if the parameter $\delta_v$ is ordinary for each $v \in S_{p,L}$.
\end{defn}

\begin{remark}
	If $(\pi,\chi)  \in \cR\cA_{n,L}$ and $\chi$ is ordinary, then \cite[Theorem 2.4]{jackreducible} implies that $r_{\pi,\iota}|_{G_{L_{\wv}}}$ has a non-critical triangulation with parameter $\delta_{v}$ for each $v \in S_{p,L}$.
\end{remark}

\begin{defn}
	A point $z \in \cE_{n,L}(\overline{\tau})^{\mathrm{irred}}(\Qpbar)$ is said to be \emph{non-critical} if for each $v \in S_{p,L}$, $r_z|_{G_{L_{\wv}}}$ admits a non-critical triangulation of parameter $\delta_v$, where $\jmath\nu(z) = \prod_{v \in S_{p,L}}\delta_v$. 
\end{defn} 
If $z$ is a non-critical point, the parameters $\delta_v = \delta_{v,alg}\delta_{v,sm}$ are necessarily locally algebraic and the local representations $r_z|_{G_{L_{\wv}}}$ are de Rham. The (distinct) labelled Hodge--Tate weights correspond to the strictly dominant character $\delta_{v,alg}$. (See \cite[Proposition 2.3.4]{bellaiche_chenevier_pseudobook} and \cite[Proposition 2.6]{Hel16}.)

We now prove a simple classicality result for non-critical points, a variant of \cite[Lemma 2.30]{New21a}. 

\begin{prop}\label{prop:noncritical implies classical}
	Let $z \in \cE_{n,L}(\overline{\tau})^{\mathrm{irred}}(\Qpbar)$ be a non-critical point with $\delta:=\jmath\nu(z)\in\cT_{n,L}^{reg}(\Qpbar)$ a regular parameter. Then $z \in Z_{n,L}$.
\end{prop}
Before the proof of this proposition, we give two preliminary lemmas on $(\varphi,\Gamma_{L_{\wv}})$-modules. The role of these is to show that a non-critical $z$ as in the statement of the proposition cannot have a companion point $z'$ with $r_{z'} \cong r_z$ and $\jmath_n(\nu(z')) = \delta'_{alg}\delta_{sm}$ with the same smooth part as $\delta$ but $\delta'_{alg}\ne \delta_{alg}$. This is a small part of the companion points conjectures of Breuil \cite[Conj.~6.6]{Bre15}  and Hansen \cite[\S6.2]{Han17}. Under a stronger regularity hypothesis on $\delta$, it would essentially follow from \cite[Thm.~4.2.3]{bhs-localmodel}. We make use of the Robba ring $\cR_{E,L_{\wv}}$ and its $(\varphi,\Gamma_{L_{\wv}})$-modules, as recalled in \cite[\S2.3.1]{New21a}.

\begin{lemma}\label{lem:no change of weight companions for noncritical parameters}
	Let $v \in S_{p,L}$, and let $D$ be a $(\varphi,\Gamma_{L_{\wv}})$-module over $\cR_{E,L_{\wv}}$ admitting a triangulation $\Fil_\bullet$ of locally algebraic parameter  $\delta = \delta_{alg}\delta_{sm}\in\cT_{n,L_{\wv}}^{reg}(E)$.
	
	Suppose that for each $\tau \in \Hom_{\Qp}(L_{\wv},E)$ we have a strictly increasing sequence of integers $\wt_\tau(\delta_1)<\wt_\tau(\delta_2)< \cdots \wt_{\tau}(\delta_n)$ (i.e.~$\delta_{alg}$ is strictly dominant). 
	
	Suppose that $D$ has a triangulation  $\Fil'_\bullet$ with a locally algebraic parameter $\delta' = \delta'_{alg}\delta_{sm}$ whose smooth part coincides with that of $\delta$. Then $\delta' = \delta$. 
\end{lemma}
\begin{proof}
	Considering the Sen weights of $D$, we see that $\delta'_{alg} = w\delta_{alg}$ for an element $w$ of the Weyl group of $(\Res_{L_{\wv}/\Qp} \GL_n)\times_{\Qp} E$. Suppose, for a contradiction, that $w \ne 1$. Choose $i \ge 1$ such that $\wt_{\tau}(\delta_j) = \wt_{\tau}(\delta'_j)$ for all $\tau$ and all $j < i$ and $\wt_{\tau}(\delta_i) \le \wt_{\tau}(\delta'_i)$ for all $\tau$ with strict inequality for some $\tau = \tau_0$. 
	
	We now reduce to the case $i=1$. If $i>1$, then $\delta_1 = \delta'_1$. The fact that $\delta \in\cT_{n,L_{\wv}}^{reg}(E)$ implies that the cohomology group $H^0_{\varphi,\gamma_{L_{\wv}}}(D(\delta_1^{-1}))$ of the Herr complex (\cite[Definition 2.3.3]{Ked14}) is one-dimensional (cf.~\cite[Lemma 2.4]{New21a}). In particular, there is a unique submodule of $D$ which is isomorphic to $\cR_{E,L_{\wv}}(\delta_1)$. We deduce that $\Fil_1 = \Fil'_1$. Considering $D/\Fil_1$ and inducting on $n$, we may assume that $i=1$.
	
	The algebraic character $\delta_1'/\delta_1$ satisfies \begin{equation}\label{eq:regular change triangulation}\delta_1'/\delta_1 = \prod_{\tau}\tau(x)^{a_{\tau}}\end{equation} with $a_{\tau} = \wt_{\tau}(\delta_1) - \wt_{\tau}(\delta'_1)$. We have $a_\tau \le 0$ for all $\tau$ and $a_{\tau_0} < 0$. Suppose that $\delta'_1/\delta_j = \prod_{\tau}\tau(x)^{b_{\tau}}$ for non-negative integers $b_\tau$, for some $1 < j \le n$. Combining with (\ref{eq:regular change triangulation}) contradicts the regularity of $\delta$. We deduce using \cite[Proposition 6.2.8]{Ked14} that $H^0_{\varphi,\gamma_{L_{\wv}}}(\delta_j(\delta'_1)^{-1}) = 0$ for all $1\le j \le n$. This contradicts the assumption that $\delta'_1$ is part of a trianguline parameter for $D$. 	 
\end{proof}

The next lemma is a small refinement
 of \cite[Thm.~6.3.13]{Ked14} (only the `moreover' part is new).
\begin{lem}\label{lem:triangulation jump at critical}
	Let $X$ be a rigid analytic space over $E$. Let $D$ be a densely
	pointwise strictly trianguline $(\varphi,\Gamma_{L_{\wv}})$-module over $\cR_{X,L_{\wv}}$ of rank $n$ with parameter $\delta \in \cT_{n,L_{\wv}}(X)$.
Then for any $z \in X$ the specialization $D_z$ is trianguline with a parameter $\delta'$, where $\delta'_i(x) = \delta_{i,z}(x)\prod_{\tau} \tau(x)^{a_{i,z,\tau}}$ for integers $a_{i,z,\tau}\in \ZZ$. 
	
	Moreover, if $\delta' \ne \delta_z$ and $i_0$ is minimal such that $\delta'_{i_0} \ne \delta_{i_0,z}$, then $a_{i_0,z,\tau}\in \ZZ_{\le 0}$ for all $\tau$. 
\end{lem}
\begin{proof}
	Everything except for the `moreover' part is contained in \cite[Thm.~6.3.13]{Ked14}, so we fix the point $z \in X$. As in the proof of \emph{loc.~cit.}, we can apply \cite[Cor.~6.3.10]{Ked14} and assume that $X$ is reduced and irreducible and $D$ admits an increasing filtration $\Fil_\bullet$ by $(\varphi,\Gamma_{L_{\wv}})$-modules such that
	\begin{itemize}
		\item Each $\mathrm{Gr}_i(D)$ embeds $(\varphi,\Gamma_{L_{\wv}})$-equivariantly into $\cR_{X,L_{\wv}}(\delta_i)\otimes_{X} \cL_i$ for a line bundle $\cL_i$ over $X$, with cokernel of the embedding killed by a power of $t$. For $i=1$ this embedding is an isomorphism.
	\end{itemize}  
	We can even assume that $X = \mathrm{Sp}(A)$ is a smooth affinoid curve, pulling back along a map from a smooth curve whose image contains $z$ but which is not entirely contained in the `bad locus' $Z$ of  \cite[Cor.~6.3.10]{Ked14}. Shrinking $X$, we may assume that the bad locus is just $\{z\}$. Each specialization $\Fil_{i,z}$ is a rank $i$ $(\varphi,\Gamma_{L_{\wv}})$-module. The triangulation of $D_z$ with parameter $\delta'$ is given by taking the saturation $\Fil_i':=\Fil_{i,z}[\frac{1}{t}]\cap D_z$ for each $i$.
	
	If $\delta'_1 = \delta_{1,z}$, then $\Fil_1' = \Fil_{1,z}$. In this case, we claim that $D/\Fil_1$ is a $(\varphi,\Gamma_{L_{\wv}})$-module over $\cR_{A,L_{\wv}}$ of rank $n-1$. We recall (\cite[Defn.~2.2.12]{Ked14}) that this means that $D/\Fil_1$ is the extension of scalars of a finite projective  $(\varphi,\Gamma_{L_{\wv}})$-module over $\cR^{r_0}_{A,L_{\wv}}$ for some $r_0 > 0$. We will argue using the $\cR_{A,L_{\wv}}$-module duality $M^\vee := \Hom_{\cR_{A,L_{\wv}}}(M,\cR_{A,L_{\wv}})$ on the category of $(\varphi,\Gamma_{L_{\wv}})$-modules. 
	
	We have a map $\alpha: D^\vee \to (\Fil_1)^\vee$ such that the specialization at every point of $\mathrm{Sp}(A)$ is surjective; for the point $z$ this is a consequence of our assumption that $\Fil_1' = \Fil_{1,z}$, i.e.~that $\Fil_{1,z}$ is saturated in $D_z$, so $D_z/\Fil_{1,z}$ is projective over $\cR_{k(z),L_{\wv}}$. It follows from \cite[Lemma 2.2.11]{Ked14} that $\alpha$ is surjective. Using \cite[Lemma 2.2.9]{Ked14}, we descend $\alpha$ to a surjective map $\alpha^{r_0}: (D^\vee)^{r_0} \to ((\Fil_1)^\vee)^{r_0}$ of $(\varphi,\Gamma_{L_{\wv}})$-modules over $\cR^{r_0}_{A,L_{\wv}}$ for some $r_0 > 0$. The kernel $K:=\ker(\alpha^{r_0})$ is finite projective over $\cR^{r_0}_{A,L_{\wv}}$, since it is the kernel of a surjective map of finite projective modules. Hence $D/\Fil_1 = \left(K\otimes_{\cR^{r_0}_{A,L_{\wv}}}\cR_{A,L_{\wv}}\right)^\vee$ is a $(\varphi,\Gamma_{L_{\wv}})$-module.

	We can now induct on $n$ and reduce to the case where $\delta'_1 \ne \delta_{1,z}$. But then $\Fil'_1$ contains $\Fil_{1,z} \cong  \cR_{k(z),L_{\wv}}(\delta_{1,z})$ as a submodule, which implies our claim (using \cite[Proposition 6.2.8]{Ked14} again).
\end{proof}

\begin{proof}[Proof of Proposition \ref{prop:noncritical implies classical}]
	In the proof, assuming for contradiction that $z$ is not classical, we first find a companion point to $z$ with a non-dominant locally algebraic weight. Lemma \ref{lem:triangulation jump at critical} gives trianguline parameters $\delta'_v$ for $r_z|_{G_{L_{\wv}}}$ to which we apply Lemma \ref{lem:no change of weight companions for noncritical parameters} in order to obtain a contradiction.
	
	After extending $E$, we may assume that $z \in \cE_{n,L}(\overline{\tau})(E)$ 
	and $r_z$ takes values in $\GL_n(E)$. Let $\eta = 
	\nu(z)\delta_{B_n} = \eta_{alg}\eta_{sm}$, with $\eta_{alg}$ 
	dominant algebraic (since $\delta_{alg}$ is strictly dominant) and 
	$\eta_{sm}$ smooth. By the construction of 
	$\cE_n$, we have a non-zero space of morphisms
	
	\[0 \ne \Hom_{\prod_{v \in S_p}T_n(L_\wv)}\left( \eta, 
	J_{B_n}\left(\widetilde{S}(U_{n,L}^p,E)^{an}[\psi^*(z)]\right)\right).\] Now 
	we use some of the work of Orlik--Strauch \cite{orlik-strauch}, with 
	notation as in \cite[\S2]{Bre15}. We denote by $\gog_n$ the $\Qp$-Lie algebra of $\prod_{v 
		\in S_p}G_n(L^+_v) \cong \prod_{v \in S_p} \GL_n(L_\wv)$ and denote by $\overline{\mathfrak{b}}_n \subset 
	\gog_n$ the lower 
	triangular Borel. We define a locally analytic representation of $\prod_{v 
		\in S_p}G_n(L_v^+)$ (see \cite[Thm.~2.2]{Bre15} for the definition of 
	the functor $\cF_{\overline{B}_n}^{G_n}$):
	\[\cF_{\overline{B}_n}^{G_n}(\eta\delta_{B_n}^{-1}) := 
	\cF_{\overline{B}_n}^{G_n} 
	\left(\left(U(\gog_{n,E})\otimes_{U(\overline{\mathfrak{b}}_{n,E})}\eta_{alg}^{-1}
	\right)^\vee,\eta_{sm}\delta_{B_n}^{-1}\right).\] 
	
	The involutive functor $(\cdot)^\vee$ appearing here is exact and contravariant, and does not change the isomorphism class of irreducibles (cf.~\cite[\S3.2]{humphreys-bgg}).
	
	We write $V(\eta_{alg})$ for the algebraic representation of $\prod_{v \in S_p}G_n(L_v^+)$ with highest weight $\eta_{alg}$. The first two steps in the BGG resolution of the dual representation  $V(\eta_{alg})^\ast$ are given by the exact sequence\footnote{We use $(\cdot)^\ast$ to denote the dual representation, to avoid confusion with the involution $(\cdot)^\vee$.}
	\[\bigoplus_{\substack{w \in W\\\ell(w)=1}}U(\gog_{n,E})\otimes_{U(\overline{\mathfrak{b}}_{n,E})}(w\cdot\eta_{alg})^{-1} \to U(\gog_{n,E})\otimes_{U(\overline{\mathfrak{b}}_{n,E})}\eta_{alg}^{-1} \to V(\eta_{alg})^\ast \to 0.\]
	Here $w$ is an element of the Weyl group of 
	$(\Res_{L^+/\Q}G_n)\times_\Q E$, acting by the `dot action' defined using the half-sum of $(\Res_{L^+/\Q}B_n)\times_\Q E$-positive roots. Note that $(w\cdot\eta_{alg})^{-1} = w\bar{\cdot}(\eta_{alg}^{-1})$ where $\bar{\cdot}$ denotes the dot action defined using the half-sum of $(\Res_{L^+/\Q}\overline{B}_n)\times_\Q E$-positive roots. 
	
	The bifunctor $\cF_{\overline{B}_n}^{G_n}$ is exact and contravariant in its first argument. Applying \cite[Thm.~2.2(iii)]{Bre15} with $Q = G_n$, we obtain $\cF_{\overline{B}_n}^{G_n}(V(\eta_{alg})^\ast,\eta_{sm}\delta_{B_n}^{-1}) = V(\eta_{alg})\otimes \Ind_{\overline{B}_n}^{G_n}\eta_{sm}\delta_{B_n}^{-1}$.  We therefore obtain an exact sequence 
	\begin{equation}\label{eq:analytic BGG}\bigoplus_{\substack{w \in W\\\ell(w)=1}}\cF_{\overline{B}_n}^{G_n}((w\cdot\eta_{alg})\eta_{sm}\delta_{B_n}^{-1}) \to \cF_{\overline{B}_n}^{G_n}(\eta\delta_{B_n}^{-1}) \to V(\eta_{alg})\otimes \Ind_{\overline{B}_n}^{G_n}\eta_{sm}\delta_{B_n}^{-1} \to 0.\end{equation}
	
	By \cite[Thm.~4.3]{Bre15}, there is a non-zero space of morphisms 
	  \[0 \ne \Hom_{\prod_{v \in	S_p}G_n(L^+_v)}\left(\cF_{\overline{B}_n}^{G_n}(\eta\delta_{B_n}^{-1}),\widetilde{S}(U_{n,L}^p,E)^{an}[\psi^*(z)]\right).\] The twist by $\delta_{B_n}^{-1}$ comes from comparing the formula \cite[(18)]{Bre15} with \cite[Definition 3.4.1]{MR2292633}. The reader may find it helpful to compare with Bernstein's second adjunction; this shows that if $\chi$ is an accessible refinement of $\pi$ and $r_{N_n}\pi_v$ denotes the normalized smooth Jacquet module, then \[\Hom_{G_n(L^+_{v})}(\Ind_{\overline{B}_n}^{G_n}\left((\iota\chi_v)\delta_{B_n}^{-1/2}\right), \pi_v) = \Hom_{T_n(L^+_{v})}(\iota\chi_v, r_{N_n}\pi_v)\ne 0.\]  
	
	To complete the proof of classicality of $z$, we will show that
	\begin{equation}\label{eqn:nocompanion}\Hom_{\prod_{v \in	S_p}G_n(L^+_v)}\left(\cF_{\overline{B}_n}^{G_n}((w\cdot\eta_{alg})\eta_{sm}\delta_{B_n}^{-1}),\widetilde{S}(U_{n,L}^p,E)^{an}[\psi^*(z)]\right) = 0\end{equation} when $w \in W$ with $w\ne 1$. We only need the vanishing \eqref{eqn:nocompanion} for $w$ a simple reflection, but the same proof works for general $w$. Combined with (\ref{eq:analytic BGG}), we deduce that any map $\cF_{\overline{B}_n}^{G_n}(\eta\delta_{B_n}^{-1})\to\widetilde{S}(U_{n,L}^p,E)^{an}[\psi^*(z)]$ factors through a locally algebraic quotient of the source. The proof ends in exactly the same way as \cite[Lemma 2.30]{New21a}.
	
	Suppose, for a contradiction, that we have a non-zero map \[\cF_{\overline{B}_n}^{G_n}((w\cdot\eta_{alg})\eta_{sm}\delta_{B_n}^{-1})\to \widetilde{S}(U_{n,L}^p,E)^{an}[\psi^*(z)].\] Breuil's adjunction \cite[Thm.~4.3]{Bre15} gives us a non-zero map of $\prod_{v \in S_p}T_n(L_\wv)$-representations $(w\cdot\eta_{alg})\eta_{sm} \to 
	J_{B_n}\left(\widetilde{S}(U_{n,L}^p,E)^{an}[\psi^*(z)]\right)$, and therefore a point $z' \in \cE_{n,L}(\overline{\tau}_n)(E)$ with $r_z \cong r_{z'}$ and $\nu(z') = \nu(z)(w\cdot\eta_{alg})\eta_{alg}^{-1}$. We have $\jmath_n(\nu(z')) = (w\delta_{alg})\delta_{sm}$. Choose $v \in S_p$ with $w_v \ne 1$. 
	
	Numerically non-critical classical points accumulate at $z'$, since it has locally algebraic weight. As in the first paragraph of the proof of \cite[Lemma 2.30]{New21a}, we obtain a densely pointwise strictly trianguline family of $(\varphi,\Gamma_{L_{\wv}})$-modules over an affinoid neighbourhood of $z'$ in $\cE_{n,L}(\overline{\tau}_n)$. The parameter of this family specializes to $(w\delta_{alg,v})\delta_{sm,v}$ at $z'$. We apply Lemma \ref{lem:triangulation jump at critical} to deduce that $r_z|_{G_{L_{\wv}}}$ has a triangulation with locally algebraic parameter $\delta' = \delta'_{alg}\delta_{sm,v}$. Applying Lemma \ref{lem:no change of weight companions for noncritical parameters}, we are done as long as $\delta'_{alg} \ne \delta_{alg,v}$. 
	If $\delta'_{alg} = \delta_{alg,v}$, then the minimal $i$ such that $(w\delta_{alg,v})_{i} \ne (\delta_{alg,v})_{i}$ coincides with the minimal $i$ such that $(w\delta_{alg,v})_{i} \ne (\delta'_{alg})_{i}$.
	Denote this common minimum by $i_0$. Since $\delta_{alg}$ is strictly dominant, we have $(\delta_{alg,v})_{i_0}(x)/(w\delta_{alg,v})_{i_0}(x) = \prod_{\tau}\tau(x)^{a_{\tau}}$ for non-negative integers $a_\tau$. On the other hand, Lemma \ref{lem:triangulation jump at critical} tells us that  $(\delta'_{alg})_{i_0}(x)/(w\delta_{alg,v})_{i_0}(x)$ has the same form, with non-positive integers $a_\tau$. We deduce (again using the assumption that $\delta'_{alg} = \delta_{alg,v}$) that $\delta_{alg,v} = w\delta_{alg,v}$, which is a contradiction ($\delta_{alg,v}$ is strictly dominant and $w_v \ne 1$).
\end{proof}

\subsection{Eigenvarieties and functoriality}
We are now ready to consider the map between spaces of pseudocharacters arising from `base change of symmetric power' functoriality. We let $\overline{\tau}_2: G_{K,S}\to k$ be a conjugate self-dual pseudocharacter of dimension $2$. Let $\overline{\tau}_{n,L} = \Sym^{n-1}\overline{\tau}_2|_{G_{L,S_L}}$; then $\overline{\tau}_{n,L}$ is a conjugate self-dual pseudocharacter of dimension $n$. Taking symmetric powers and restricting to $G_L$ determines a morphism $\sigma_{n,L,\Gal}: \cX_{ps,2} \to \cX_{ps,n,L}$. On the other hand, we can define a map $\sigma_{n,p,L}: \cT_{2}\to \cT_{n,L}$  by the formula
\[(\delta_{v,1},\delta_{v,2})_{v\in S_p} \mapsto \left((\delta_{w|_{K^+},1}^{n-1},\delta_{w|_{K^+},1}^{n-2}\delta_{w|_{K^+},2},\ldots,\delta_{w|_{K^+},2}^{n-1})\circ \mathbf{N}_{L_{\ww}/K_{\ww|_{K}}}\right)_{w \in S_{p,L}}.\]

Putting the two maps together, we have $\sigma_{n,L} = \sigma_{n,L,\Gal}\times\sigma_{n,p,L}: \cX_{ps,2} \times \cT_2 \to \cX_{ps,n,L}\times \cT_{n,L}$ and a diagram
\[\xymatrix{\cE_{2}(\overline{\tau}_2)\ar[dr]^{\sigma_{n,L}\circ i_2} & & \cE_{n,L}(\overline{\tau}_{n,L})\ar[dl]^{i_n}\\ &\cX_{ps,n,L}\times \cT_{n,L} &}\]

\begin{defn}
	Let $(\pi,\chi) \in \cR\cA_2$. We say that $\chi$ is $(n, L)$-regular  if for each place $v \in S_p$ and $w|v$ in $L^+$, $\left(\frac{\chi_{v,1} \circ \mathbf{N}_{L_{\ww} / K_{\wv}}}{\chi_{v,2} \circ \mathbf{N}_{L_{\ww} / K_{\wv}}}\right)^i \neq 1$ for each $i = 1, \dots, n-1$. 
\end{defn}
\begin{lem}\label{lem:regular after BCSym}
	\begin{enumerate}
	\item Suppose $(\pi,\chi) \in \cR\cA_2$ and $\chi$ is $(n,L)$-regular. Let $z = \gamma(\pi,\chi) \in \cE_{2}(\Qpbar)$. Then $\sigma_{n,p,L}(\jmath\nu(z)) \in \cT_{n,L}(\Qpbar)$ is regular.
	
	\item Suppose $(\pi,\chi) \in \cR\cA_2$ and $\chi$ is ordinary. Then $\chi$ is $(n,L)$-regular for all $n$ and $L$.
\end{enumerate}
\end{lem}
\begin{proof}
	If $\delta = \jmath\nu(z)$, then $\delta_{v,1}/\delta_{v,2}$ differs from $\chi_{v,1}/\chi_{v,2}$ by an algebraic character (cf.~\cite[Lemma 2.8]{New21a}). Noting that for $w|v$ in $L$ and $1 \le i < j \le n$, \[\sigma_{n,p,L}(\delta)_{w,i}/\sigma_{n,p,L}(\delta)_{w,j} = \left(\frac{\delta_{v,1} \circ \mathbf{N}_{L_{\ww} / K_{\wv}}}{\delta_{v,2} \circ \mathbf{N}_{L_{\ww} / K_{\wv}}}\right)^{j-i}\] we see that these ratios cannot be algebraic characters when $\chi$ is $(n,L)$-regular. When $\chi$ is ordinary, $\chi_{v,1}(\varpi_{\wv})/\chi_{v,2}(\varpi_{\wv})$ has non-zero $p$-adic valuation.
\end{proof}

Now we can state our fundamental analytic continuation results.
\begin{theorem}[Analytic continuation of functoriality, non-ordinary case]\label{thm:an cont unitary non-ord}
Let $(\pi,\chi) \in \cR\cA_2$ satisfy $\tr \rbar_{\pi,\iota} = \overline{\tau}$, and let $z = \gamma(\pi,\chi) \in \cE_2(\overline{\tau})(\Qpbar)$. Suppose that:
\begin{enumerate}
	\item $\chi$ is $(n,L)$-regular and $z$ is non-critical.
	\item There exists $(\pi_{n,L},\chi_n) \in \cR\cA_{n,L}$ such that $(\sigma_{n,L}\circ i_2)(z) = i_n(z_n)$, where $z_n = \gamma_n(\pi_{n,L},\chi_n)$.
	\item For each $v \in S_p$, the Zariski closure of $r_{\pi,\iota}(G_{K_{\wv}})$ (in $\GL_2/\Qpbar$) contains $\SL_2$.
\end{enumerate}
Then each irreducible component $\cC$ of $\cE_2(\overline{\tau})_{\bC_p}$ containing $z$ satisfies $(\sigma_{n,L}\circ i_2)(\cC)\subset i_n(\cE_{n,L}(\overline{\tau}_{n,L})_{\bC_p})$.
\end{theorem}
\begin{proof}
	The proof is identical to that of \cite[Theorem 2.24]{New21a}. In \emph{loc.~cit.} we assumed that $\chi$ is a numerically non-critical refinement, but we only use the consequence that the associated trianguline parameter $\delta$ for $r_{\pi,\iota}$ is non-critical. The assumption that $\chi$ is $(n,L)$-regular is what we need to ensure that $\delta_{n,L} = \sigma_{n,L}(\delta)$ is a regular trianguline parameter for $r_{\pi_{n,L},\iota}$ (since $z$ is non-critical, it will also be a non-critical parameter).
\end{proof}

\begin{theorem}[Analytic continuation of functoriality, ordinary case]\label{thm:an cont unitary ord}
	Let $(\pi,\chi) \in \cR\cA_2$ satisfy $\tr \rbar_{\pi,\iota} = \overline{\tau}$, and let $z = \gamma(\pi,\chi) \in \cE_2(\overline{\tau})(\Qpbar)$. Suppose that:
	\begin{enumerate}
		\item $\chi$ is ordinary.
		\item There exists $(\pi_{n,L},\chi_n) \in \cR\cA_{n,L}$ such that $(\sigma_{n,L}\circ i_2)(z) = i_n(z_n)$, where $z_n = \gamma_n(\pi_{n,L},\chi_n)$.
		\item The Zariski closure of $r_{\pi,\iota}(G_{K})$ contains $\SL_2$.
	\end{enumerate}
	Then each irreducible component $\cC$ of $\cE_2(\overline{\tau})_{\bC_p}$ containing $z$ satisfies $(\sigma_{n,L}\circ i_2)(\cC)\subset i_n(\cE_{n,L}(\overline{\tau}_{n,L})_{\bC_p})$.
\end{theorem}
\begin{proof}
	This proof is identical to that of \cite[Theorem 2.27]{New21a}. 
\end{proof}

To complete our results in the unitary setting, we combine the preceding theorems with the classicality criterion given by Proposition \ref{prop:noncritical implies classical}.

\begin{cor}\label{cor: main U2 application}
	Let $(\pi, \chi), (\pi', \chi') \in \mathcal{RA}_2$, and let $z, 
z' \in \cE_2(\Qpbar)$ be the corresponding classical points of the eigenvariety. 
Suppose that one of the following two sets of conditions are satisfied:
\begin{enumerate}
	\item The points $z$ and $z'$ are non-critical and the refinements $\chi$ and $\chi'$ are $(n,L)$-regular.
	\item For each $v \in S_p$, the Zariski closures of the images of $r_{\pi, \iota}|_{G_{K_\wv}}$ and $r_{\pi',\iota}|_{G_{K_\wv}}$ contain $\SL_2$.
	\item The points $z$ and $z'$ lie on a common irreducible component of $\cE_{2,\Cp}$;
\end{enumerate}	or
\begin{enumerate}
	\item[(1\textsuperscript{ord})] The refinements $\chi$ and $\chi'$ are ordinary.
	\item[(2\textsuperscript{ord})] The Zariski closures of the images of (the global representations) $r_{\pi, \iota}$ and $r_{\pi', \iota}$ contain $\SL_2$.
	\item[(3\textsuperscript{ord})] The points $z$ and $z'$ lie on a common irreducible component of	$\cE_{2,\Cp}$.
\end{enumerate}	
Suppose moreover that there exists an automorphic representation $\pi_{n,L}$ of 
$G_n(\bA_{L^+})$ such that 
\[ \Sym^{n-1} r_{\pi, \iota}|_{G_L} \cong r_{\pi_{n,L}, \iota}. \] Then there exists an automorphic representation $\pi'_{n,L}$ of 
$G_n(\bA_{L^+})$ such that \[ \Sym^{n-1} r_{\pi', \iota}|_{G_L} \cong r_{\pi'_{n,L}, \iota}. \]
\end{cor}
\begin{proof}
This is proved in the same way as \cite[Corollary 2.28]{New21a}. Choose $U_{n,L} \subset G_n(\bA_{L^+}^\infty)$ so that $(\pi_{n,L}^\infty)^{U_{n,L}^{p}} \neq 
0$ and take $\overline{\tau}_2 = \tr \overline{r}_{\pi, \iota}$. As in \emph{loc.~cit.}, there is an accessible refinement $\chi_n$ of $\pi_{n,L}$ such that $(\sigma_{n,L} 
\circ i_2)(z_2) = i_n(\gamma_n(\pi_{n,L},\chi_n)) \in i_n(\cE_{n,L}(\overline{\tau}_{n,L})(\Qpbar))$. Applying Theorem \ref{thm:an cont unitary non-ord} or Theorem \ref{thm:an cont unitary ord}, we deduce the existence of a point $z_n' \in \cE_{n,L}(\overline{\tau}_{n,L})$ with $i_n(z_n') = \sigma_{n,L}(i_2(z'))$. Our assumption that $z'$
is non-critical and $\chi'$ is $(n,L)$-regular (or that $\chi'$ is ordinary) implies that $z_n'$ is non-critical with a regular parameter. Proposition \ref{prop:noncritical implies classical} tells us that $z_n'$ is classical, which gives us the desired $\pi'_{n,L}$.
\end{proof}
\begin{rem}
	The assumption that $z$ and $z'$ lie in a common irreducible component implies that one of the refinements $\chi, \chi'$ is ordinary if and only if the other is (cf.~\cite[Lemma 2.26]{New21a}). We have imposed the same technical assumptions on the two points $z$ and $z'$ for simplicity. In the non-ordinary case, we only need to assume that the global representation $r_{\pi',\iota}$ has big image; in the ordinary case, we don't need any big image assumption on $r_{\pi',\iota}$.
\end{rem}
\subsection{Back to $\GL_2$}
Let $F$ be a totally real field. Let $N \geq 1$ 
be an integer and let $p$ be a prime, and let $\cE_{0}$ denote the associated 
cuspidal tame level $N$ eigencurve. Fix an isomorphism $\iota : 
\overline{\bQ}_p \to \bC$. Following the notation of \cite{New21a}, we write 
$\mathcal{RA}_0$ for the set of pairs $(\pi_0, \chi_0)$, where $\pi_0$ is a 
cuspidal, regular algebraic automorphic representation of $\GL_2(\bA_\bQ)$ of weight $k \ge 2$ such that $(\pi_0^{\infty, p})^{U_1(N)^p} \neq 0$, and $\chi_0 = \chi_{0, 1} \otimes \chi_{0, 2}$ is an accessible refinement of $\pi_0$ such that $\chi_{0, 1}$ is 
\emph{unramified}. Any such pair determines an associated classical point $z_0 \in \cE_0(\overline{\bQ}_p)$.

We fix an integer $n \geq 2$.
\begin{defn}
Let $(\pi_0, \chi_0) \in \cR\cA_0$. We say that $\chi_0$ is $(n, F)$-regular if for each place $v | 
p$ of $F$, $\left(\frac{\chi_{0, 1} \circ \mathbf{N}_{F_v / \bQ_p}}{\chi_{0, 2} \circ \mathbf{N}_{F_v / \bQ_p}}\right)^i \neq 1$ for each $i = 1, \dots, n-1$. 
\end{defn}
When $p$ splits in $F$, $(n, F)$-regularity is the `$n$-regular'  condition of \cite{New21a}. If $\chi_0$ is $M$-regular for some $M > n[F: \bQ]$, then $\chi_0$ is $(n, F)$-regular. 
	\begin{proposition}\label{prop_analytic_continuation}
Let $(\pi_0, \chi_0)$, $(\pi'_0, \chi'_0) \in \mathcal{RA}_0$, and let $z_0, z_0'$ be the associated classical points. Suppose that one of the following two sets of conditions is satisfied:
\begin{enumerate}
\item The Zariski closures of $r_{\pi_0, \iota}(G_{\bQ_p})$ and $r_{\pi_0', 
\iota}(G_{\bQ_p})$ contain $\SL_2$.\label{assm:pbig}
\item $\chi_{0}$ and $\chi'_{0}$ are $(n, F)$-regular.
\end{enumerate}
or
\begin{enumerate}
\item [$(1^{ord})$] The refinements $\chi_0$ and $\chi_0'$ are ordinary (concretely, this means that $v_p(\chi_{0,1}(p)) = v_p(\chi'_{0,1}(p)) = -1/2$).
\item [$(2^{ord})$] $\pi_0$ and $\pi_0'$ are not CM.
\end{enumerate}
Suppose moreover that the points $z_0, z_0'$ lie on a common irreducible component of $\cE_{0, \bC_p}$. Then $\mathrm{BC}_{F / \bQ}(\Sym^{n-1} \pi_0)$ exists if and only if $\mathrm{BC}_{F / \bQ}(\Sym^{n-1} \pi'_0)$ exists.
\end{proposition}
\begin{proof}

To prove the proposition, we follow closely the proof of 
\cite[Thm.~2.33]{New21a}. We let $K/\Q$ 
be the soluble Galois CM extension denoted $F$ in \emph{loc.~cit.}~and set $L^+ = K^+F$, $L = KL^+$. In particular, $L/F$ is a soluble Galois CM extension with $L/L^+$ everywhere unramified, and in which every prime of $F$ dividing $Np$ splits. 
By Lemma \ref{lem_OK_after_base_change}, it suffices to show that $\mathrm{BC}_{L^+ / \bQ}(\Sym^{n-1}\pi_0')$ exists.

Let $S$ denote the set of places of $K^+$ dividing $Np$ and fix factorisations $v = \wv\wv^c$ for $v \in S$. We have defined unitary groups over $K^+$ in \S\ref{subsec:unitary}. Our level subgroup $U_2 = \prod_v U_{2,v}\subset G_2(\bA_{K^+}^\infty)$ is hyperspecial maximal compact at $v \not\in S$ and corresponds to the standard level subgroup $U_1(N)_l \subset \GL_2(\Z_l)$ at places $v \in S$ with residue characteristic $l$ (these places are split over $\Q$).

As in the proof of \cite[Thm.~2.33]{New21a}, soluble base change to $K$, twisting by a character and descending to the unitary group $U_2/K^+$ \cite[Theorem 1.4]{New21a} gives us two classical points 
$z = \gamma_2(\pi,\chi), z' = \gamma_2(\pi',\chi')$ in $\cE_2(\Qpbar)$ contained in a single geometrically irreducible 
component of $\cE_2$ with the following properties:
\begin{enumerate}
	\item $i_2(z) = (\tr r_z, \delta_z)$ and $i_2(z') = (\tr r_{z'}, 
	\delta_{z'})$, where $r_z$ and $r_{z'}$ are absolutely irreducible 
	representations $G_{K,S} \to 
	\GL_2(\Qpbar)$ which are conjugate self-dual twists of 
	$r_{\pi_0,\iota}|_{G_{K}}$ and $r_{\pi'_0,\iota}|_{G_{K}}$ 
	respectively. 
	\item $z$ and $z'$ are non-critical.
	\item $\chi$ and $\chi'$ are $(n,L)$-regular.
	\item The Zariski closures of the images of $r_{\pi, \iota}$ and $r_{\pi', \iota}$ contain $\SL_2$.
	\item Either $\chi$ and $\chi'$ are ordinary, or for each $v \in S_p$, the Zariski closures of the images of $r_{\pi, \iota}|_{G_{K_\wv}}$ and $r_{\pi',\iota}|_{G_{K_\wv}}$ contain $\SL_2$.
\end{enumerate} 
The first property is satisfied by construction. The second property follows from non-criticality of the points $z_0$ and $z'_0$; the only critical classical points of $\cE_0$ come from the non-ordinary refinement of an ordinary $\pi$ (cf.~\cite[Example 2.10]{New21a}). The third property follows from our $(n,F)$-regularity assumption, which is automatic in the ordinary case (cf.~Lemma \ref{lem:regular after BCSym}). In the ordinary case, the fourth property follows from the non-CM assumption. The fifth property is immediate from our assumptions. 

Now assume that $\mathrm{BC}_{F / \bQ}(\Sym^{n-1} \pi_0)$ exists. We can apply a further soluble base change to $L$. As a consequence, there exists an automorphic representation $\pi_{n,L}$ of $G_n(\bA_{L^+})$ such that $\Symm^{n-1}r_{\pi,\iota}|_{G_L} \cong r_{\pi_{n,L},\iota}$. Applying Corollary \ref{cor: main U2 application}, we deduce that there exists an automorphic representation $\pi'_{n,L}$ of $G_n(\bA_{L^+})$ such that $\Symm^{n-1}r_{\pi',\iota}|_{G_L} \cong r_{\pi'_{n,L},\iota}$. The representation $r_{\pi'_{n,L},\iota}$ is an absolutely irreducible twist of $\Symm^{n-1}r_{\pi'_0,\iota}|_{G_L}$. Base change to $\GL_n/L$ \cite[Theorem 1.2]{New21a}, undoing the twist,
and descending to $L^+$ \cite[Lemma 1.5]{blght} shows that $\mathrm{BC}_{L^+ / \bQ}(\Sym^{n-1}\pi_0')$ exists.
\end{proof}

\section{Existence of $n$-regular lifts}

The following theorem is a generalisation of \cite[Proposition 8.3]{New21a}.
\begin{theorem}\label{thm_congruence_to_n_regular_liftings}
Let $\pi$ be a cuspidal, regular algebraic automorphic representation of $\GL_2(\bA_\bQ)$ of weight $k \geq 2$. Let $p$ be a prime such that $p > k$ and $\pi_p$ is unramified, and let $\iota : \overline{\bQ}_p \to \bC$ be an isomorphism such that $\overline{r}_{\pi, \iota}|_{G_{\bQ(\zeta_p)}}$ is irreducible. Suppose given the following data:
\begin{enumerate}
\item A finite set $S$ of primes, containing the set $S(\pi)$ of primes at which $\pi$ is ramified, but not containing $p$.
\item For each $l \in S$, a lifting $\rho_l : G_{\bQ_l} \to \GL_2(\overline{\bZ}_p)$ of $\overline{r}_{\pi, \iota}|_{G_{\bQ_l}}$.
\item An integer $M \geq 1$.
\end{enumerate}  
Then we can find a finite set $Q$ of prime numbers and another cuspidal, regular algebraic automorphic representation $\pi'$ of $\GL_2(\bA_\bQ)$ of weight $k$ satisfying the following conditions:
\begin{enumerate}
\item There is an isomorphism $\overline{r}_{\pi', \iota} \cong \overline{r}_{\pi, \iota}$.
\item $Q \cap S = \emptyset$ and $\pi'$ is unramified outside $S \cup Q$.
\item For each $q \in Q$, $q \equiv 1 \text{ mod }p$ and $\rec_{\bQ_q}(\pi'_q) \cong \chi_{q, 1} \oplus \chi_{q, 2}$ for characters $\chi_{q, 1}, \chi_{q, 2} : \bQ_q^\times \to \bC^\times$ such that $\chi_{q, 1}|_{\bZ_q^\times}$ has $p$-power order greater than $M$ and $\chi_{q, 1} \chi_{q, 2}$ is unramified.
\item For each $l \in S$, $r_{\pi', \iota}|_{G_{\bQ_l}} \sim \rho_l$.
\item For any $l \in S \cup Q$ such that $\pi'_l$ is not supercuspidal, $\pi'_l$ is $M$-regular (i.e.\ admits an $M$-regular refinement).
\end{enumerate} 
\end{theorem}
\begin{proof}
We use a variant of the argument, based on Taylor--Wiles--Kisin patching, used to prove \cite[Proposition 8.3]{New21a}. Choose a coefficient field $E / \bQ_p$ such that each representation $r_{\pi, \iota}$ and $\rho_l$ takes values in $\GL_2(E)$. Let $\overline{\rho} = \overline{r}_{\pi, \iota}$, $\chi = \det r_{\pi, \iota}$. Enlarging $E$ further, we can assume that the eigenvalues of all elements of $\overline{\rho}(G_\bQ)$ lie in $k$. Consider the global deformation problem as in \emph{loc. cit.}
\[ \cS = (\overline{\rho}, \chi, S \cup \{ p \}, \{ \cO \}_{v \in S \cup \{ p \}}, \{ \cD_v \}_{v \in S \cup \{ p \}}), \]
where if $v \in S$, $\cD_v$ is the functor of all liftings of $\overline{\rho}|_{G_{\bQ_v}}$ of determinant $\chi$, and if $v = p$, then $\cD_v$ is the functor of Fontaine--Laffaille liftings of $\overline{\rho}|_{G_{\bQ_p}}$ of Hodge--Tate weights $0, k-1$ and determinant $\chi$. Thus $\cD_p$ is formally smooth and if $l \in S$, then the representing object $R_l \in \cC_\cO$ of $\cD_l$ is an $\cO$-flat complete intersection over $\cO$, of relative dimension 3. 

Choose an integer $N$, divisible only by primes in $S$, such that for each $l \in S$, any lift of $\overline{\rho}|_{G_{\bQ_l}}$ has conductor dividing $N$. Choose a prime $q_a > \max(S \cup \{ p \})$ such that the functor of lifts of $\overline{\rho}|_{G_{\bQ_{q_a}}}$ of determinant $\chi$ is formally smooth (and in particular, such that any lift is unramified) -- this is possible by \cite[Lemma 11]{MR1262939}. Let $H = H^1(Y_{U_1(N q_a)}, \Sym^{k-2} \cO^2)$ (where the curve and level are as in \cite[Proposition 8.3]{New21a}), let $\bT \subset \End_\cO H$ be the commutative $\cO$-subalgebra generated by the unramified Hecke operators $T_l, S_l$ at primes $l \not\in S \cup \{p, q_a \}$, and let $\mathfrak{m} \subset \bT$ be the maximal ideal associated to the reduction modulo $p$ of the Hecke eigenvalues of $\iota^{-1} \pi^\infty$. Then $H_\mathfrak{m}$ is a non-zero $\bT_{\mathfrak{m}}$-module and there is a unique surjective homomorphism of $\cO$-algebras $R_\cS \to \bT_{\mathfrak{m}}$ sending, for each prime $l \not\in S \cup \{ p, q_a \}$, the characteristic polynomial of $\Frob_l$ in the universal deformation to $X^2 - T_l X + l^{k-1} S_l$ (justification as in the proof of \cite[Proposition 8.3]{New21a}).

Suppose given a finite set $Q$ of primes satisfying the following conditions:
\begin{itemize}
\item[(a)] $Q \cap (S \cup \{ p, q_a \}) = \emptyset$.
\item[(b)] For each $q \in Q$, $q \equiv 1 \text{ mod }p$ and $\overline{\rho}(\Frob_q)$ has distinct eigenvalues $\alpha_q, \beta_q \in k$.
\end{itemize}
In this case we define $\Delta_Q = \prod_{q \in Q} (\bZ / q \bZ)^\times(p)$ and the augmented deformation problem
\[ \cS_Q = (\overline{\rho}, \chi, S \cup \{ p \} \cup Q, \{ \cO \}_{v \in S \cup \{ p \} \cup Q}, \{ \cD_v \}_{v \in S \cup \{ p \} \cup Q}), \]
where $q \in Q$ we again take the functor $\cD_q$ of all liftings of $\overline{\rho}|_{G_{\bQ_q}}$ of determinant $\chi$. Then, as described in the proof of \cite[Proposition 8.3]{New21a}, the representing object $R_{\cS_Q}$ of $\cS_Q$ has a structure of $\cO[\Delta_Q]$-algebra such that there is a canonical isomorphism $R_{\cS_Q} \otimes_{\cO[\Delta_Q]} \cO \cong R_{\cS}$, and there is moreover a $R_{\cS_Q}$-module $H_{Q, \mathfrak{m}_{Q, 1}} = H^1(Y_{U_1(N q_a) \cap U_2(Q)}, \Sym^{k-2} \cO^2)_{\mathfrak{m}_{Q, 1}}$, free over $\cO[\Delta_Q]$, together with an isomorphism $H_{Q, \mathfrak{m}_{Q, 1}} \otimes_{\cO[\Delta_Q]} \cO \cong H_{\mathfrak{m}}$ of $R_\cS$-modules.

Let $A^S_\cS = \widehat{\otimes}_{l \in S} R_l$ denote the completed tensor product over $\cO$ of the local lifting rings $R_l$. According to the Taylor--Wiles--Kisin patching construction, we can find the following data:
\begin{itemize}
\item[(c)] An integer $q_0 \geq 0$.
\item[(d)] For each $L \geq 1$, a set $Q_L$ of primes satisfying conditions (a) and (b) above, such that $|Q_L| = q_0$ and for each $q \in Q_L$, $q \equiv 1 \text{ mod }p^L$.
\item[(e)] An extension of the natural map $A^S_\cS \to R^S_{\cS_{Q_L}}$ to a surjective algebra homomorphism $A^S_\cS \llbracket X_1, \dots, X_g \rrbracket \to R^S_{\cS_{Q_L}}$, where $g = q_0 + |S| - 1$.
\end{itemize}
Let $S_\infty = \cO \llbracket \bZ_p^{q_0} \rrbracket$, and fix for each $L \geq 1$ a surjection $\bZ_p \to (\bZ / q \bZ)^\times(p)$, and therefore a surjection $S_\infty \to \cO[\Delta_Q]$ (using the ordering of the primes $q_i \in Q_L$ by absolute value). Fix $a \geq 1$ such that $p^a > M$, and for $i = 1, \dots, q_0$, let $x_i = [(0, \dots, 0, 1, 0, \dots, 0)]^{p^a} - 1 \in S_\infty$, where $(0, \dots, 0, 1, 0, \dots, 0) \in \bZ_p^{q_0}$ has 1 in the $i^\text{th}$ co-ordinate and 0 elsewhere. Then $x = \prod_{i=1}^{q_0} x_i$ is a non-zero element of the integral domain $S_\infty$.

Let $S_0 \subset S$ denote the set of primes such that $\rho_l$ is not irreducible. For each $l \in S_0$, we fix a Frobenius lift $\phi_l \in G_{\bQ_l}$, and let $\delta_l = (\operatorname{disc} \det(X - \Sym^M r_l(\phi_l))) \subset R_l$, where $r_l : G_{\bQ_l} \to \GL_2(R_l)$ is the universal lifting. Let $\delta = \prod_l \delta_l \in A_\cS^S$. After possibly enlarging $\cO$ at the start of the proof, we can assume that the irreducible components of each ring $R_l$ ($l \in S$) are geometrically irreducible, and that the irreducible components of $A_\cS^S \llbracket X_1, \dots, X_g \rrbracket \to R_\infty$ are therefore determined by tuples of components of the rings $R_l$ ($l \in S$) (see \cite[Lemma 3.3]{blght}). We select for each $l \in S$ a minimal prime $Q_l \subset R_l$ containing the point corresponding to the fixed lifting $\rho_l$, and let $Q \subset A_\cS^S \llbracket X_1, \dots, X_g \rrbracket $ be the minimal prime corresponding to this tuple of choices. We claim that the image of $\delta \in A_\cS^S / Q$ is non-zero. It suffices to show that for each $l \in S_0$, the image of $\delta_l \in R_l / Q_l$ is non-zero;  equivalently, that there exists a homomorphism $R_l/Q_l \to \overline{\bQ}_p$ under which $\delta_l$ has non-zero image. If $\operatorname{disc} \det(X - \Sym^M \rho_l(\phi_l)) \neq 0$, then this is immediate. Otherwise, we can assume that this discriminant is 0, implying that the Frobenius semi-simplification of $\rho_l$ is a direct sum of two characters. If $\rho_l$ is already semi-simple, then we can write down a deformation of $\rho_l \otimes_\cO E$ to $E \llbracket X \rrbracket$ over which $\delta_l$ is non-zero, as in the proof of \cite[Lemma 8.4]{New21a}, and apply \cite[Proposition 2.1.5]{Gee11} to conclude that $\delta_l$ is not zero in $R_l / Q_l$. If $\rho_l$ is not semisimple then it must be a character twist of an unramified representation, and $R_l / Q_l$ is identified (after twist) with an unramified lifting ring. In this case it's easy to check directly that $\delta_l$ is not zero in $R_l / Q_l$.

Let us say that a regular algebraic automorphic representation $\pi'$ of $\GL_2(\bA_\bQ)$ that contributes to one of the spaces $H^1(Y_{U_1(N q_a) \cap U_2(Q_L)}, \Sym^{k-2} \cO^2)_{\mathfrak{m}_{Q, 1}}$ $(L \geq 1)$ is `permissible' if  for each $l \in S$, $r_{\pi', \iota}|_{G_{\bQ_l}}$ determines a point of $R_l / Q_l$, and `bad' if it is permissible and either there exists $q \in Q_L$ such that $\rec_{\bQ_q}(\pi')(I_{\bQ_q})$ has order dividing $p^a$, or if there exists $l \in S$ such that $\rho_l$ is reducible and $\det(X - \Sym^M r_{\pi', \iota}(\phi_l))$ has a repeated root. If $\pi'$ is `bad' then $x\delta$ lies in the kernel of the map $R_{\cS_{Q_L}}^S \to \overline{\bQ}_p$ associated to $r_{\pi', \iota}$. We need to show that there is a permissible $\pi'$ which is not bad.

Let $r_\cS : G_\bQ \to \GL_2(R_\cS)$ be a representative of the universal deformation. For each $L \geq 1$, choose a representative $r_{\cS_{Q_L}} : G_\bQ \to \GL_2(R_{\cS_{Q_L}})$ of the universal deformation lifting $r_\cS$. Let $\cT = \cO \llbracket T_1, \dots, T_{4 |S|} \rrbracket / (T_1)$. The choice of $r_\cS$ and $r_{\cS_{Q_L}}$ determines compatible isomorphisms $R^S_{\cS} \cong R_{\cS} \widehat{\otimes}_\cO \cT$ and $R^S_{\cS_{Q_L}} \cong R_{\cS_{Q_L}} \widehat{\otimes}_\cO \cT$. We set $H^S_{Q_L} = H_{Q_L, \mathfrak{m}_{Q, 1}} \widehat{\otimes}_\cO \cT$, which becomes a $R_{\cS_{Q_L}}^S$-module. Let $Q_{(1)}, \dots, Q_{(s)}$ be the minimal prime ideals of $A_\cS^S \llbracket X_1, \dots, X_g \rrbracket$ other than $Q$. Then we can find an element $f \in \cap_{i=1}^s Q_{(i)} - Q$.

We claim that if the conclusion of the theorem is false (in particular, if every permissible $\pi$ is bad), then $f x \delta \cdot H^S_{Q_L} = 0$. Let $\bT_{Q_L} \subset \End_\cO H_{Q_L, \mathfrak{m}_{Q_L, 1}}$ denote the $\cO$-subalgebra generated by the unramified Hecke operators $T_l, S_l$ at primes $l \not\in S \cup \{ p, q_a \} \cup Q$. Then $\bT_{Q_L}$ is a finite flat local $\cO$-algebra which is moreover reduced. It follows that $\bT_{Q_L} \widehat{\otimes}_\cO \cT$ is reduced; this ring may be identified with the $\cT$-subalgebra of $\End_\cT(H^S_{Q_L})$ generated by the same unramified Hecke operators. To prove the claim, it therefore suffices to show that the image of $fx\delta$ in $\bT_{Q_L} \widehat{\otimes}_\cO \cT$ is 0. Since this ring is reduced and the closed points of $\Spec \bT_{Q_L} \widehat{\otimes}_\cO \cT[1/p]$ are Zariski dense, it even suffices to show that the image of $f x \delta$ under the map $R^S_{\cS_{Q_L}} \to \overline{\bQ}_p$ associated to any automorphic representation $\pi'$ which contributes to $H_{Q_L}$, together with a choice of framing of $r_{\pi', \iota}$, is 0. To see this, we split into cases. If $\pi'$ is not permissible, then there is $l \in S$ such that $r_{\pi', \iota}|_{G_{\bQ_l}}$ does not determine a point of $R_l/Q_l$, implying that the map $A^S_\cS \to \overline{\bQ}_p$ factors through $A^S_{\cS} / Q_{(i)}$ for some $i = 1, \dots, s$. In particular, the image of $f$ is 0. If $\pi'$ is permissible but bad, then as we have seen our assumption implies that $x\delta$ has image 0. This establishes the claim.

To finish the proof, we will derive a contradiction from the assumption that $f x \delta \cdot H^S_{Q_L} = 0$ for every $L \geq 1$. If this is the case, then after patching we obtain the following data:
\begin{itemize}
\item An object $R_\infty \in \cC_\cO$ together with $\cO$-algebra homomorphisms $S_\infty \widehat{\otimes}_\cO \cT \to R_\infty$ and $A_\cS^S \llbracket X_1, \dots, X_g \rrbracket \to R_\infty$, of which the second is surjective.
\item An $R_\infty$-module $H_\infty$, which is free over $S_\infty \widehat{\otimes}_\cO \cT$.
\item A homomorphism $R_\infty \otimes_{S_\infty \widehat{\otimes}_\cO \cT} \cO \to R_\cS$, and an isomorphism $H_\infty \otimes_{S_\infty \widehat{\otimes}_\cO \cT} \cO \cong H_{\mathfrak{m}}$ of $R_\infty$-modules.
\item $f x \delta H_\infty = 0$.
\end{itemize}
We can lift the map $S_\infty \widehat{\otimes}_\cO \cT \to R_\infty$ to a map $S_\infty \widehat{\otimes}_\cO \cT  \to A_\cS^S \llbracket X_1, \dots, X_g \rrbracket$. The ring $A_\cS^S \llbracket X_1, \dots, X_g \rrbracket $ is a reduced, equidimensional, flat $\cO$-algebra of Krull dimension $\dim S_\infty \widehat{\otimes}_\cO \cT$. By \cite[Lemma 2.3]{tay}, we can conclude that the support of $H_\infty$, as $A_\cS^S \llbracket X_1, \dots, X_g \rrbracket$-module, is a union of irreducible components of $\Spec A_\cS^S \llbracket X_1, \dots, X_g \rrbracket$.
In fact, $H_\infty$ is a faithful $A_\cS^S \llbracket X_1, \dots, X_g \rrbracket$-module; this follows from the existence of automorphic liftings of $\overline{\rho}$ meeting any given component of each lifting ring $R_l$ $(l \in S$) (cf. \cite[Proposition 3.1.7]{Gee11}; this reference treats the potentially Barsotti--Tate case, but the proof in our case is completely analogous and in fact easier, since we do not allow any possibility for variation in the behaviour at the prime $p$). 

Since $f x \delta H_{\infty, (Q)} = 0$ and $f \delta$ is a unit in $A_\cS^S \llbracket X_1, \dots, X_g \rrbracket_{(Q)}$, we find that $x H_{\infty, (Q)} = 0$. Since multiplication by $x$ is injective on the free $S_\infty$-module $H_\infty$, and localisation is exact, we conclude that $H_{\infty, (Q)} = 0$. This contradicts the fact that $H_{\infty}$ is a faithful $A_\cS^S \llbracket X_1, \dots, X_g \rrbracket$-module, and so $H_{\infty, (Q)} \neq 0$. This contradiction concludes the proof.
\end{proof}

\begin{lemma}\label{lem_removing_ramification_non-ordinary_case}
Let $p$ be a prime number, let $N \geq 1$ be an integer which is prime to $p$, and let $M \geq 1$ be an integer. Let $\iota : \overline{\bQ}_p \to \bC$, and let $(\pi, \chi) \in \mathcal{RA}_0$ (the set of refined, cuspidal, regular algebraic automorphic representations of $\GL_2(\bA_\bQ)$ of tame level divinding $N$). Suppose that the following conditions are satisfied:
\begin{enumerate} 
\item The Zariski closure of $r_{\pi, \iota}(G_{\bQ_p})$ contains $\SL_2$.
\item There is a finite set of prime numbers $S$ such that for each $l \in S$, $\pi_l$ is $M$-regular.
\end{enumerate}
Let $z \in \cE_{0, \bC_p}$ be the point of the (tame level $N$, $p$-adic) eigencurve corresponding to $(\pi, \chi)$, and let $\cC \subset \cE_{0, \bC_p}$ be an irreducible component containing $z$. Then we can find another point $z' \in \cC$, corresponding to a pair $(\pi', \chi') \in \mathcal{RA}_0$, with the following properties:
\begin{enumerate}
\item The Zariski closure of $r_{\pi', \iota}(G_{\bQ_p})$ contains $\SL_2$. In particular, the refinement $\chi'$ is numerically non-critical.
\item For each $l \in S \cup \{ p \}$, $\pi'_l$ is $M$-regular.
\item $\pi'_p$ is unramified, and for any prime $l \neq p$, $\pi_l$ and $\pi'_l$ are inertially equivalent, in the sense that if $(r, N) = \rec_{\bQ_l}(\pi_l)$ and $(r', N') = \rec_{\bQ_l}(\pi'_l)$, then there is an isomorphism $(r'|_{I_{\bQ_l}}, N') \cong (r|_{I_{\bQ_l}}, N)$.
\end{enumerate}
 \end{lemma}
  \begin{proof}
 According to (a very slight variation of) \cite[Lemma 2.35]{New21a}, there is a Zariski closed subset $\cZ \subset \cE_{0}$ such that any $x \in \cE_{0}(\overline{\bQ}_p) - \cZ(\overline{\bQ}_p)$ has the following properties:
 \begin{itemize}
 \item The semisimple representation $\rho_x : G_\bQ \to \GL_2(\overline{\bQ}_p)$ is irreducible, and the Zariski closure of its image contains a conjugate of $\SL_2$.
 \item For each $l \in S - \{ p \}$, if $\rho_x|_{G_{\bQ_l}}^{ss}$ is reducible with constituents $\chi_1, \chi_2$, then $(\chi_1 / \chi_2)^i \neq 1$ for $i = 1, \dots, M$.
 \end{itemize}
 By assumption, $z \not\in \cZ(\overline{\bQ}_p)$, so $\cZ_{\bC_p} \cap \cC$ is a proper Zariski closed subset of $\cC$. The image $\kappa(\cC) \subset \cW_{0, \bC_p}$ is Zariski open, so we can find a point $z_0 \in \cC$ such that $\kappa(z_0)$ is of the form $y \mapsto y^{k_0 - 2}$ for some $k_0 \in \bZ_{\geq 2}$. We may choose an affinoid neighbourhood $U_0$ of $z_0$ in $\cE_{0, \bC_p}$ such that the image $\kappa(U_0)$ is an affinoid subset of $\cW_{0, \bC_p}$ and $\kappa|_{U_0}$ is finite flat onto its image. Then $U_0 \cap \cZ_{\bC_p}$ is finite, so we can choose another point $z_1 \in (U_0 \cap \cC) - \cZ_{\bC_p}$ such that $\kappa(z_1)$ is of the form $y \mapsto y^{k_1 - 2}$ for some $k_1 \in \bZ_{\geq 2}$; choosing $k_1$ sufficiently large, we can moreover assume that $s(z_1) < (k_1 - 2)/2$; then $z_1$ has small slope, so corresponds to a pair $(\pi', \chi') \in \mathcal{RA}_0$. We claim that this pair has the desired properties. The representation $\pi'_p$ is unramified since the only other possibility is that it is an unramified twist of the Steinberg representation, which would force $s(z) = (k_1 - 2)/2$. It has two refinements of distinct slopes, so $\pi'_p$ is $M'$-regular for every $M' \geq 1$. Since the component $\cC$ is non-ordinary (it contains the point $z$), $\pi'$ is not $\iota$-ordinary, and so $r_{\pi', \iota}|_{G_{\bQ_p}}$ is irreducible. Then \cite[Lemma 3.5]{New21a} shows that the Zariski closure of $r_{\pi', \iota}(G_{\bQ_p})$ in fact contains $\SL_2$. We have avoided the set $\cZ$, so if $l \in S - \{ p \}$ then $\pi'_l$ is $M$-regular by construction. Finally, if $l \in S - \{ p \}$ then we at least have $r_{\pi, \iota}|_{I_{\bQ_l}}^{ss} \cong r_{\pi', \iota}|_{I_{\bQ_l}}^{ss}$ (there are only a finite number of possibilities for the character $\tr r_{\pi', \iota}|_{I_{\bQ_l}}$, and $\cC$ is irreducible). If one of $\rec_{\bQ_l}(\pi_l)$ or $\rec_{\bQ_l}(\pi'_l)$ has non-zero $N$, then so does the other (as if $N \neq 0$ then the relation $\chi_1 / \chi_2 = | \cdot |^{\pm 1}$ is forced to hold on the whole of the irreducible component $\cC$), and $\pi_l, \pi_l'$ are inertially equivalent. If both representations have $N = 0$ then $r_{\pi, \iota}|_{I_{\bQ_l}}$, $r_{\pi', \iota}|_{I_{\bQ_l}}$ are both semisimple and we're done in this case also.
     \end{proof}

\begin{lemma}\label{lem_removing_ramification_ordinary_case}
Let $p$ be a prime number, let $N \geq 1$ be an integer which is prime to $p$, and let $M \geq 1$ be an integer. Let $\pi$ be a regular algebraic, cuspidal automorphic representation of $\GL_2(\bA_\bQ)$ without CM, and let $\iota : \overline{\bQ}_p \to \bC$. Suppose that the following conditions are satisfied:
\begin{enumerate} 
\item $\pi$ is $\iota$-ordinary. 
\item There is a finite set of prime numbers $S$ such that for each $l \in S$, $\pi_l$ is $M$-regular.
\end{enumerate}
Let $z \in \cE_{0, \bC_p}$ be the point of the (tame level $N$, $p$-adic) eigencurve corresponding to $\pi$, and let $\cC \subset \cE_{0, \bC_p}$ be an irreducible component containing $z$. Then we can find another point $z' \in \cC$, corresponding to a regular algebraic, cuspidal automorphic representation of $\GL_2(\bA_\bQ)$ without CM, with the following properties:
\begin{enumerate}
\item $\pi'$ is $\iota$-ordinary.
\item For each $l \in S \cup \{ p \}$, $\pi'_l$ is $M$-regular.
\item $\pi'_p$ is unramified, and for any prime $l \neq p$, $\pi_l$ and $\pi'_l$ are inertially equivalent, in the sense that if $(r, N) = \rec_{\bQ_l}(\pi_l)$ and $(r', N') = \rec_{\bQ_l}(\pi'_l)$, then there is an isomorphism $(r'|_{I_{\bQ_l}}, N') \cong (r|_{I_{\bQ_l}}, N)$.
\end{enumerate}
 \end{lemma}
  \begin{proof}
 The proof is very similar to the proof of Lemma \ref{lem_removing_ramification_non-ordinary_case}, although simpler since the ordinary component $\cC$ must map surjectively to an irreducible component of $\cW_{0, \bC_p}$ and any point $z' \in \cC$ such that $\kappa(z')$ has the form $y \mapsto y^{k'-2}$, $k' \in \bZ_{ \geq 3}$, is necessarily classical, corresponding to an $\iota$-ordinary automorphic representation of $\GL_2(\bA_\bQ)$ that is unramified at the prime $p$. The remainder of the argument is the same.
  \end{proof}
\begin{corollary}\label{cor_removing_a_prime_of_ramification}
Let $F$ be a totally real number field, and let $n \geq 2$, $M \geq \max(3,n [F : \bQ])$ be integers. Let $\pi$ be a regular algebraic, cuspidal automorphic representation of $\GL_2(\bA_\bQ)$ without CM, and let $S$ be a finite set of prime numbers containing the set of primes at which $\pi$ is ramified. Suppose that for each $l \in S$, $\pi_l$ is $M$-regular.

Let $p \in S$. Then we can find another regular algebraic, cuspidal automorphic representation $\pi'$ of $\GL_2(\bA_\bQ)$ without CM and with the following properties:
\begin{enumerate}
\item $\pi'_p$ is unramified. For each prime $l \neq p$, $\pi_l$ and $\pi'_l$ are inertially equivalent.
\item For each $l \in S$, $\pi'_l$ is $M$-regular. 
\item $\mathrm{BC}_{F / \bQ}(\Sym^{n-1} \pi)$ exists if and only if $\mathrm{BC}_{F / \bQ} (\Sym^{n-1} \pi')$ does.
\end{enumerate} 
\end{corollary}
\begin{proof}
Fix an isomorphism $\iota : \overline{\bQ}_p \to \bC$. Choose an accessible refinement $\chi$ of $\pi$, taking the ordinary one if it exists. After possibly replacing $\pi$ by a character twist, we can assume that $(\pi, \chi) \in \mathcal{RA}_0$. If there is no ordinary refinement, then the $M$-regularity of $\pi$ (with $M \geq 3$) implies that the Zariski closure of $r_{\pi, \iota}(G_{\bQ_p})$ contains a conjugate of $\SL_2$ (\cite[Lemma 3.5]{New21a}). Let $\cC$ be an irreducible component of $\cE_{0, \bC_p}$, the $p$-adic eigencurve of tame level equal to the prime-to-$p$ level of $\pi$, containing the point corresponding to $(\pi, \chi)$. We then apply Lemma \ref{lem_removing_ramification_non-ordinary_case} or Lemma \ref{lem_removing_ramification_ordinary_case} to construct $\pi'$ and Proposition \ref{prop_analytic_continuation} to conclude that $\mathrm{BC}_{F / \bQ}(\Sym^{n-1} \pi)$ exists if and only if $\mathrm{BC}_{F / \bQ} (\Sym^{n-1} \pi')$ does.
\end{proof} 

\section{Symmetric power functoriality for residually dihedral forms}

In this section we prove a rather general result concerning the existence of base change lifts of symmetric powers of automorphic representations having a sufficiently robust dihedral mod $p$ Galois representation for some prime $p$, which replaces the intricate \cite[Corollary 7.3]{New21a}. This is made possible by the automorphy lifting theorems proved in the residually reducible case in \cite{jackreducible, All19} and the level-raising theorems proved in \cite{Tho22}.
The base change lifts constructed here will serve as the `seed points' for the analytic continuation in the proofs of our main theorems in the following sections. 
\begin{theorem}\label{thm_residually_dihedral_symmetric_power}
	Let $F_0^+$ be a totally real number field, let $(\pi, \psi)$ be a cuspidal, regular algebraic, polarized
automorphic representation of $\GL_2(\bA_{F_0^+})$, let $p \geq 5$ be a prime, and let $\iota : \overline{\bQ}_p \to \bC$ be an isomorphism. Let $n \geq 2$ be an integer, and let $F^+ / F_0^+$ be a finite totally real extension. Suppose that the following conditions are satisfied:
	\begin{enumerate}
		\item $\pi$ is $\iota$-ordinary.		
		\item There is a CM quadratic extension $F / F^+$ and an isomorphism $\overline{r}_{\pi, \iota}|_{G_{F^+}} \cong \Ind_{G_{F}}^{G_{F^+}} \overline{\chi}$, for a character $\overline{\chi} : G_F \to \overline{\bF}_p^\times$ such that $\overline{\chi} / \overline{\chi}^c|_{G_{F(\zeta_p)}}$ has order greater than $n(n-1)$. Moreover, $[F(\zeta_p) : F] = (p-1)$ and $p > n$.
				\item There is a place $v_0 \nmid p$ of $F_0^+$ such that $\pi_{v_0}$ is a twist of the Steinberg representation.
	\end{enumerate}
	Then $\mathrm{BC}_{F^+ / F_0^+} \Sym^n \pi$ exists: there is a cuspidal, regular algebraic, polarizable automorphic representation $\Pi$ of $\GL_{n}(\bA_{F^+})$ such that $\Sym^{n-1} r_{\pi, \iota}|_{G_{F^+}} \cong r_{\Pi, \iota}$. 
\end{theorem}
\begin{proof}
Let $q$ denote the residue characteristic of the place $v_0$. Our local hypothesis at $v_0$ implies that the Zariski closure of the image of $r_\iota(\pi)$ contains $\SL_2$,
and therefore that $\Sym^{n-1} r_{\pi, \iota}$ is irreducible, even after restriction to any closed, finite index subgroup of $G_{F^+_0}$. We have $\det r_{\pi, \iota} = \epsilon^{-1} r_{\psi, \iota}$. We are free to replace $F^+$ by a soluble, totally real extension, and can therefore assume that the following additional conditions are satisfied:
\begin{itemize}
\item The extension $F / F^+$ is everywhere unramified.
\item $F$ contains an imaginary quadratic field $F_0$ in which $p$ and $q$ split. We fix a 
place $v_1$ of $F^+$ lying above $v_0$. We assume that $q_{v_1} \equiv 1 \text{ mod }p$, that $\overline{r}_{\pi, \iota}|_{G_{F_{v_1}}}$ is unramified, and that $\overline{r}_{\pi, \iota}(\Frob_{v_1})$ is scalar. 
\item There exists an everywhere unramified Hecke character $\Omega : \bA_F^\times \to \bC^\times$ of type $A_0$ such that $\Omega\Omega^c = \psi \circ \mathbf{N}_{F / F^+}$. Let $\omega = r_{\Omega, \iota}$. 
\item There exists an everywhere unramified Hecke character $X : \bA_F^\times \to \bC^\times$ of type $A_0$ such that $\chi = r_{X, \iota}$ lifts $\overline{\chi}$ and $\chi \chi^c = \epsilon^{-1} r_{\psi, \iota}|_{G_F}$.
\item There exists a CM subfield $F' \subset F$ such that $[F : F'] = 2$ and the infinity types of $\Omega, X$ descend to $F'$.
\end{itemize} 
(The existence of such characters, perhaps with ramification, can be deduced from \cite[Lemma A.2.5]{BLGGT}; we can then choose the soluble extension that replaces $F^+$ so that they become everywhere unramified and so that $F$ has a subfield $F'$ of the given type. These conditions will be used to check that the hypotheses of \cite[Theorem 6.11]{Tho22} are satisfied.) To prove the theorem it suffices, by \cite[Theorem 7.5]{Tho22}, to find an $\iota$-ordinary regular algebraic, cuspidal, conjugate self-dual automorphic representation $\Pi_{n}$ of $\GL_{n}(\bA_F)$ and a place $w_1 | v_1$ of $F$ such that $\Pi_{n, w_1}$ is a twist of the Steinberg representation and there is an isomorphism 
\[ \overline{r}_{\Pi_n, \iota} \cong \overline{\omega}^{1-n} \otimes \Sym^{n-1} \overline{r}_{\pi, \iota}|_{G_F} = \overline{\omega}^{1-n} \otimes \oplus_{i=0}^{n-1} \overline{\chi}^{n-1-i} \overline{\chi}^{c, i}. \]
Indeed, we could then conclude the automorphy of $\omega^{1-n} \otimes \Sym^{n-1} r_{\pi, \iota}|_{G_F}$ and hence, by soluble descent, that of $\Sym^{n-1} r_{\pi, \iota}|_{G_{F^+}}$. We explain why the hypotheses  of \cite[Theorem 7.5]{Tho22} on the residual representation are satisfied. Let $\overline{\mu}_i = \omega^{1-n} \overline{\chi}^{n-1-i} \overline{\chi}^{c, i}$. Then $\overline{\mu}_i \overline{\mu}_i^c = \epsilon^{1-n}$. If $0 \leq i < j \leq n-1$ then the ratio $\overline{\mu}_i / \overline{\mu}_j = (\overline{\chi} / \overline{\chi}^c)^{j-i}$ has order greater than $n$, so $\overline{\rho} = \oplus_{i=0}^{n-1} \overline{\mu}_i$ is primitive, by \cite[Lemma 5.1]{New21a}, and the characters $\overline{\mu}_i$ are pairwise distinct. If $\tau_0 \in G_F$
is an element such that $\epsilon(\tau_0)^2 \neq 1$, and $w$ is a place of $F$ unramified in $\overline{\rho}$ and split over $F^+$ such that $(\overline{\rho} \oplus \epsilon)(\Frob_w) = (\overline{\rho} \oplus \epsilon)(\tau_0 \tau_0^c)$, then $H^2(F_w, \ad \overline{\rho}(1)) = 0$; such an element $\tau_0$ exists since $p-1 > 2$. 

We will in fact show by induction on $k \geq 2$ that for each $k = 2, \dots, n$, there is an $\iota$-ordinary regular algebraic, cuspidal, conjugate self-dual automorphic representation $\Pi_k$ of $\GL_{k}(\bA_F)$ satisfying the following conditions:
\begin{itemize}
\item There is an isomorphism $\overline{r}_{\Pi_k, \iota} \cong \overline{\omega}^{1-k} \otimes \Sym^{k-1} \overline{r}_{\pi, \iota}|_{G_F}$.
\item There is a place $w_1 | v_1$ of $F$ such that $\Pi_{k, w_1}$ is an unramified twist of the Steinberg representation.
\item $\Pi_{k}$ is unramified at each $q$-adic place of $F$ not lying above $v_1$, and at each place of $F$ which is inert over $F^+$. 
\item For each embedding $\tau : F \to \overline{\bQ}_p$, we have 
\[ \mathrm{HT}_\tau(r_{\Pi_k, \iota}) = \cup_{i = 0}^{k-1} \mathrm{HT}_\tau(\omega^{1-k} \chi^{k-1-i} \chi^{c, i}). \]
\end{itemize} 
 For the base case $k = 1$, we apply \cite[Theorem 6.11]{Tho22} with $\pi_1 = \Omega^{-1} X | \cdot|^{1/2}$ and $\pi_2 = \Omega^{-1} X^c | \cdot |^{1/2}$. For the induction step, we apply \cite[Theorem 6.11]{Tho22} with $\pi_1 = \Pi_k \otimes \Omega^{-1} X | \cdot |^{1/2}$ and $\pi_2 = \Omega^{-k} X^{c, k} | \cdot |^{(k-1)/2}$. The level-raising congruence and genericity hypotheses of \cite[Theorem 6.11]{Tho22} are satisfied at each stage because there is an isomorphism
\[ \overline{r}_{\pi_1 \boxplus \pi_2, \iota} \cong \overline{\omega}^{-(k+1)} \otimes \oplus_{i=0}^{k+1} \overline{\chi}^{k+1-i} \overline{\chi}^{c, i}. \]
Condition (1) of \cite[Theorem 4.5]{Tho22} is satisfied because of our assumption on the existence of the subfield $F' \subset F$. This concludes the proof.
\end{proof}

\section{Proof of main theorem -- `everywhere finite slope' case}

Let $F$ be a totally real number field. Fix an integer $n \geq 2$. In this section, we will prove the following result, intermediate to our main theorem:
\begin{theorem}\label{thm_intermediate_step}
Let $\pi$ be a regular algebraic, cuspidal automorphic representation of $\GL_2(\bA_\bQ)$ without CM. Suppose that for each prime number $p$, $\pi_p$ is not supercuspidal. Then $\mathrm{BC}_{F / \bQ}(\Sym^{n-1} \pi)$ exists.
\end{theorem}
Fix an integer $M \geq \max(3, n[F: \bQ])$. In fact, we will prove by induction on $N \geq 2$ the following statement:
\begin{itemize}
\item[$(H_N)$] Let $\pi$ be a representation satisfying the hypotheses of Theorem \ref{thm_intermediate_step}. If $p > N$ is a prime number, suppose that $\pi_p$ is unramified. If $p \leq N$ is a prime number, suppose that $\pi_p$ is $M$-regular. Then $\mathrm{BC}_{F / \bQ}(\Sym^{n-1} \pi)$ exists.
\end{itemize}
\begin{proposition}
Suppose that $(H_N)$ holds for all $N \geq 2$. Then Theorem \ref{thm_intermediate_step} is true.
\end{proposition}
\begin{proof}
Let $\pi$ be a representation satisfying the hypotheses of Theorem \ref{thm_intermediate_step}. Choose an integer $N$ such that $\pi_p$ is unramified if $p > N$. After replacing $\pi$ by a character twist, we can assume that it is of weight $k \geq 2$. Let $t$ be a prime number satisfying the following conditions:
\begin{itemize}
\item $t > \max(N, 2(n+1), (n-1)(k-1))$.
\item There is an isomorphism $\iota_t : \overline{\bQ}_t \to \bC$ such that $\overline{r}_{\pi, \iota_t}(G_\bQ)$ contains a conjugate of $\SL_2(\bF_t)$.
\end{itemize}
By Theorem \ref{thm_congruence_to_n_regular_liftings}, we can find a regular algebraic, cuspidal automorphic representation $\pi'$ of $\GL_2(\bA_\bQ)$ of weight $k$ satisfying the following conditions:
\begin{itemize}
\item There is an isomorphism $\overline{r}_{\pi', \iota_t} \cong \overline{r}_{\pi, \iota_t}$. 
\item There is a finite set of primes $Q$ such that for each $q \in Q$, $\pi_q$ is unramified, $q > N$, and $\rec_{\bQ_q}(\pi'_q) \cong \chi_{1, q} \oplus \chi_{2, q}$ for characters $\chi_{1, q}, \chi_{2, q} : \bQ_q^\times \to \bC^\times$ such that the character $\chi_{1, q}|_{\bZ_q^\times}$ has order $p^a$ for some $p^a > M$, and $\chi_{1, q} \chi_{2, q}$ is unramified.
\item For each prime number $p \not\in Q$, we have $r_{\pi, \iota_t}|_{G_{\bQ_p}} \sim r_{\pi', \iota_t}|_{G_{\bQ_p}}$.
\item For each prime number $p \leq N$, $\pi'_p$ is $M$-regular.
\end{itemize}
Both representations $\Sym^{n-1} r_{\pi', \iota_t}$ and $\Sym^{n-1} r_{\pi, \iota_t}$ are Fontaine--Laffaille and residually irreducible, so using Lemma \ref{lem_large_image_over_F} we can apply an automorphy lifting theorem (e.g. \cite[Theorem 4.2.1]{BLGGT}) to conclude that $\mathrm{BC}_{F / \bQ}(\Sym^{n-1} \pi)$ exists if and only if $\mathrm{BC}_{F / \bQ}(\Sym^{n-1} \pi')$ does. 

Enumerate $Q = \{ q_1, \dots, q_r \}$. We claim that we can find a sequence $\pi_0, \dots, \pi_r$ of regular algebraic, cuspidal automorphic representations of $\GL_2(\bA_\bQ)$, without CM, and with the following properties:
\begin{itemize}
\item $\pi_r = \pi'$.
\item Let $S(\pi_i)$ denote the set of primes at which $\pi_i$ is ramified. Then $S(\pi_i) = S(\pi) \cup \{ q_1, \dots, q_i \}$. In particular, $S(\pi_0) = S(\pi)$.
\item For each prime number $p$ such that $p \in S(\pi_i)$ or $p \leq N$, $\pi_{i, p}$ is $M$-regular.
\item For each $i = r, \dots, 1$, the automorphic representation  $\mathrm{BC}_{F / \bQ}(\Sym^{n-1} \pi_i)$ exists if and only if $\mathrm{BC}_{F / \bQ}(\Sym^{n-1} \pi_{i-1})$ does.
\end{itemize}
This will complete the proof of the Proposition. Indeed, the hypothesis $(H_N)$ implies that $\mathrm{BC}_{F / \bQ}(\Sym^{n-1} \pi_0)$ exists, a property which can then be propagated back along the chain $\pi_0, \dots, \pi_r=\pi', \pi$. To get these representations, we simply take $\pi_r = \pi'$ and then apply Corollary \ref{cor_removing_a_prime_of_ramification} $r$ times.
\end{proof} 
\begin{proposition}\label{prop_existence_of_M_regular_form_with_base_change_symmetric_power}
Let $p$ be a prime. There exists a regular, cuspidal algebraic automorphic representation $\pi$ of $\GL_2(\bA_\bQ)$ without CM with the following properties:
\begin{enumerate}
\item $\pi$ is everywhere unramified.
\item $\pi_p$ admits a numerically non-critical, $M$-regular refinement.
\item $\mathrm{BC}_{F / \bQ}(\Sym^{n-1} \pi)$ exists. 
\end{enumerate} 
\end{proposition}
\begin{proof}
Let $K$ be an imaginary quadratic field in which $p$ is unramified, and let $X : \bA_K^\times / K^\times \to \bC^\times$ be an everywhere unramified character of type $A_0$ such that if $z \in K_\infty^\times$ and $\tau : K \to \bC$ is a fixed embedding, then $X(z) = \tau(z)^{1-k}$ for some $k \geq 2$. Let $\sigma = \mathrm{AI}_{K / \bQ}(X|\cdot|^{1/2})$ be the automorphic induction of $X|\cdot|^{1/2}$; then $\sigma$ is a regular algebraic, cuspidal automorphic representation of $\GL_2(\bA_\bQ)$ of weight $k$.

We choose a prime $t > \max(2k n(n-1), p)$ which is split in $KF$, and an isomorphism $\iota_t : \overline{\bQ}_t \to \bC$. Then $\sigma$ is $\iota_t$-ordinary. Let $\chi = r_{X, \iota_t}$, so that $r_{\sigma, \iota_t} = \Ind_{G_K}^{G_{\bQ}} \chi$. We claim that $\overline{\chi} / \overline{\chi}^c|_{G_{KF(\zeta_t)}}$ has order greater than $n(n-1)$ (in particular, that $\overline{r}_{\sigma, \iota_t}|_{G_{\bQ(\zeta_t)}}$ is absolutely irreducible). To see this, let $L / KF$ be the extension cut out by $\overline{\chi} / \overline{\chi}^c$. Since $t$ splits in $KF$, this extension is cyclic of order at least $(t-1) / (k-1) > t / k$. Considering the action of $c$ on the Galois group, we see that $L \cap KF(\zeta_t)$ is an extension of $KF$ of degree at most 2. Therefore $L(\zeta_t) / KF(\zeta_t)$ has degree greater than $t / 2k > n(n-1)$.

Now we apply Theorem \ref{thm_congruence_to_n_regular_liftings}, with $S = S(\sigma)\cup\{p,q\}$ for a prime $q$ splitting completely in the extension of $\bQ(\zeta_t)$ cut out by $\overline{r}_{\sigma, \iota_t}$. For $l \in S(\sigma)\cup\{p\}$, we choose $\rho_l$ to be any reducible lifting of $\overline{r}_{\sigma, \iota_t}|_{G_{\bQ_l}}$. We choose $\rho_q$ corresponding to the Steinberg representation. The theorem tells us that we can find another regular algebraic, cuspidal automorphic representation $\sigma_1$ of $\GL_2(\bA_\bQ)$ of weight $k$ and satisfying the following conditions:
\begin{itemize}
	\item There is an isomorphism $\overline{r}_{\sigma_1, \iota_t} \cong \overline{r}_{\sigma, \iota_t}$.
	\item $\sigma_{1,q}$ is an unramified twist of the Steinberg representation.
	\item $\sigma_{1,t}$ is unramified.
	\item For every prime number $l \in S(\sigma_1) \cup \{ p \}$, $\sigma_{1, l}$ admits an $M$-regular refinement.
\end{itemize} 
In particular, $\sigma_1$ is $\iota_t$-ordinary (since $\overline{r}_{\sigma_1, \iota_t}$ is Fontaine--Laffaille and congruent to the ordinary representation $\overline{r}_{\sigma, \iota_t}$) and has no supercuspidal local components. 

We see that $\sigma_1$ satisfies the hypotheses of Theorem \ref{thm_residually_dihedral_symmetric_power} (with respect to the extension $F / \bQ$), so we conclude that $\mathrm{BC}_{F / \bQ}(\Sym^{n-1} \sigma_1)$ exists. We can then repeatedly apply Corollary \ref{cor_removing_a_prime_of_ramification} to $\sigma_1$ to obtain an automorphic representation $\pi$ with the required properties.
\end{proof}
We actually use Proposition \ref{prop_existence_of_M_regular_form_with_base_change_symmetric_power} only in the case $p = 2$. We prove a similar result (in the case $p = 2$) in \cite{New21a}, where we used special properties of level 1 modular forms (cf. \cite[Lemma 3.4]{New21a}). The argument given here is more robust, with a view to future generalisations. 
\begin{proposition}\label{prop_base_change_for_level_1_forms}
$(H_2)$ holds.
\end{proposition}
\begin{proof}
By Proposition \ref{prop_existence_of_M_regular_form_with_base_change_symmetric_power}, we can find an everywhere unramified regular algebraic, cuspidal automorphic representation $\pi_0$ of $\GL_2(\bA_\bQ)$ such that $\pi_{0, 2}$ is $M$-regular and such that $\mathrm{BC}_{F / \bQ}(\Sym^{n-1} \pi_0)$ exists. The result then follows on combining \cite[Lemma 3.7]{New21a} with Proposition \ref{prop_analytic_continuation}. (This is the ``ping pong'' argument of \emph{op. cit.}, which relies on the explicit description of the 2-adic, tame level 1 eigencurve given in \cite{Buz05}.)
\end{proof}
\begin{proposition}
If $N \geq 2$, then $(H_N) \Rightarrow (H_{N+1})$. Therefore, Theorem \ref{thm_intermediate_step} is true. 
\end{proposition}
\begin{proof}
Suppose that $(H_N)$ holds, and let $\pi$ be a regular algebraic, cuspidal automorphic representation of $\GL_2(\bA_\bQ)$ without CM and satisfying the following additional conditions:
\begin{itemize}
\item For every prime $p$, $\pi_p$ is not supercuspidal.
\item If $p > N+1$ is prime, then $\pi_p$ is unramified.
\item If $p \leq N+1$ is prime, then $\pi_p$ is $M$-regular.
\end{itemize}
We need to show that $\mathrm{BC}_{F / \bQ}(\Sym^{n-1} \pi)$ exists. After replacing $\pi$ by a character twist, we can assume that it is of weight $k \geq 2$. We can also assume, without loss of generality, that $N+1 = p$ is prime and $\pi_p$ is ramified. We can then invoke Corollary \ref{cor_removing_a_prime_of_ramification} to find another automorphic representation $\pi'$ of $\GL_2(\bA_\bQ)$ satisfying the same conditions as $\pi$ but relative to $N$ instead of $N+1$, and moreover such that $\mathrm{BC}_{F / \bQ}(\Sym^{n-1} \pi)$ exists if and only if $\mathrm{BC}_{F / \bQ}(\Sym^{n-1} \pi')$ does. The existence of $\mathrm{BC}_{F / \bQ}(\Sym^{n-1} \pi')$ follows by the induction hypothesis $(H_N)$; and this completes the proof.
\end{proof}

\section{Proof of the main theorem -- general case}

Fix a totally real number field $F$ and integer $n \geq 2$. In this section, we will complete the proof of the following theorem:
\begin{theorem}\label{thm_existence_of_base_change_with_supercuspidal_local_components}
Let $\pi$ be a regular algebraic automorphic representation of $\GL_2(\bA_\bQ)$, without CM. Then $\mathrm{BC}_{F / \bQ}(\Sym^{n-1} \pi)$ exists.
\end{theorem}
We will prove the theorem, as in \cite{New21b}, by induction on the cardinality of the set $sc(\pi)$ of primes $p$ such that $\pi_p$ is supercuspidal, the base case $sc(\pi) = \emptyset$ having been treated already. The proof we give in the case $n > 2$ will rely on the case $n = 2$ (in which case $\Sym^{n-1} \pi = \pi$) since the `functoriality lifting theorem' proved in \cite{New21b} requires as an assumption the existence of the base change $\mathrm{BC}_{F / \bQ}(\pi)$ in order to be able to conclude the existence of $\mathrm{BC}_{F / \bQ}(\Sym^{n-1} \pi)$.

In \cite{New21b}, Theorem \ref{thm_existence_of_base_change_with_supercuspidal_local_components} in the case $F = \bQ$ is reduced to the case where $\pi$ has no supercuspidal local components by constructing a chain of congruences and using automorphy lifting theorems. Lemma \ref{lem_large_image_over_F} shows that this chain of congruences is still robust after base change to $F$, and as such the arguments here follow those of \cite[\S 3]{New21b} very closely. Only the case $p = 2$ (when Lemma \ref{lem_large_image_over_F} does not apply) needs a small bit of extra care. 

Let $\pi$ be a regular algebraic cuspidal automorphic representation of $\GL_2(\bA_\bQ)$ of weight 2, and such that $2 \not\in sc(\pi)$. Let $q, t, r$ be prime numbers. In \cite[Definition 3.6]{New21b}, we have defined what it means for $\pi$ to be `seasoned' with respect to $(q, t, r)$. This condition depends on $n$ through the inequality $t > \max(10, 8n(n-1))$. Let $N = \max(3, n)$. When we use `seasoned' below, we will mean seasoned with respect to $N$, i.e.\ we are assuming $t > \max(10, 8N(N-1))$. This is imposed because of the following result, which is \cite[Proposition 3.7]{New21b}:
\begin{prop}
Let $\pi$ be seasoned with respect to $(q, t, r)$. Then $q \in sc(\pi)$ and for each prime $p \in sc(\pi)$, there exists an isomorphism $\iota : \overline{\bQ}_p \to \bC$ such that $\overline{r}_{\pi, \iota}(G_\bQ)$ contains a conjugate of $\SL_2(\bF_{p^a})$ for some $p^a > 2N-1$. 
\end{prop}
Thus $p^a > 5$ (for any $n$). 
\begin{prop}\label{prop_putting_in_a_supercuspidal_prime}
Let $\pi$ be a regular algebraic, cuspidal automorphic representation of $\GL_2(\bA_\bQ)$ which is non-CM. Suppose that $2, 3, \not\in sc(\pi)$. Then we can find a regular algebraic, cuspidal automorphic representation $\pi'$ of $\GL_2(\bA_\bQ)$ with the following properties:
\begin{enumerate}
	\item There exist prime numbers $(q, t, r)$ such that $\pi'$ is seasoned with respect to $(q, t, r)$ and $sc(\pi') = sc(\pi) \cup \{ q  \}$.
	\item $\mathrm{BC}_{F / \bQ}(\Sym^{n-1} \pi)$ exists if and only if $\mathrm{BC}_{F / \bQ}(\Sym^{n-1} \pi')$ does.
\end{enumerate}
\end{prop}
\begin{proof}
Running through the first part of the proof of \cite[Proposition 3.9]{New21b}, we see that we can find a prime $t > \max(10, 8N(N-1),N(\pi))$ and a regular algebraic, cuspidal automorphic representation $\pi'$ of $\GL_2(\bA_\bQ)$, all satisfying the following conditions:
\begin{itemize}
\item There exists an isomorphism $\iota : \overline{\bQ}_t \to \bC$ such that $\operatorname{Proj} \overline{r}_{\pi, \iota}(G_\bQ)$ is conjugate either to $\PSL_2(\bF_t)$ or $\PGL_2(\bF_t)$. 
\item $\pi'$ is seasoned and there is an isomorphism $\overline{r}_{\pi, \iota} \cong \overline{r}_{\pi', \iota}$. 
\end{itemize}
We need to conclude that $\Sym^{n-1} r_{\pi, \iota}|_{G_F}$ is automorphic if and only if $\Sym^{n-1} r_{\pi', \iota}|_{G_F}$ is automorphic. As in the proof of \cite[Proposition 3.9]{New21b}, this follows from the automorphy lifting theorem \cite[Theorem 4.2.1]{BLGGT} -- the large image condition is satisfied over $F$ by Lemma \ref{lem_large_image_over_F} and the potential diagonalizability condition satisfied over $F$ because it is preserved under base extension.
\end{proof}
\begin{prop}\label{prop_taking_out_a_supercuspidal_prime}
Let $\pi$ be a regular algebraic, cuspidal automorphic representation of $\GL_2(\bA_\bQ)$ of weight 2. Suppose that $\pi$ is seasoned with respect to $(q, t, r)$, and let $p \in sc(\pi)$ satisfy $p \geq 5$. Then we can find a regular algebraic, cuspidal automorphic representation $\pi'$ of $\GL_2(\bA_\bQ)$ with the following properties:
\begin{enumerate}
\item $\pi'$ has weight 2 and is non-CM.
\item $sc(\pi') = sc(\pi) - \{ p \}$. If $p \neq q$, then $\pi'$ is seasoned with respect to $(q, t, r)$.
\item If $\mathrm{BC}_{F / \bQ}(\Sym^{n-1} \pi')$ exists, then so does $\mathrm{BC}_{F / \bQ}(\Sym^{n-1} \pi)$. 
\end{enumerate} 
\end{prop}
\begin{proof}
The proof of \cite[Proposition 3.10]{New21b} shows that we can find an isomorphism $\iota : \overline{\bQ}_p \to \bC$ and a regular algebraic, cuspidal automorphic representation $\pi'$ of $\GL_2(\bA_\bQ)$ with the following properties:
\begin{itemize}
\item The image of $\overline{r}_{\pi, \iota}$ contains a conjugate of $\SL_2(\bF_{p^a})$ for some $p^a > 2N-1$.
\item $\pi'$ has weight 2 and there is an isomorphism $\overline{r}_{\pi, \iota} \cong \overline{r}_{\pi', \iota}$. In particular, $\pi'$ is non-CM.
\item If $p \neq q$, then $\pi'$ is seasoned with respect to $(q, t, r)$. In any case, $sc(\pi') = sc(\pi) - \{ p \}$.
\item For each prime number $l \neq p$, we have $r_{\pi, \iota}|_{G_{\bQ_l}} \sim r_{\pi', \iota}|_{G_{\bQ_l}}$. Moreover, $\pi'$ is not $\iota$-ordinary. 
\end{itemize} 
If $n = 2$, then \cite[Theorem 3.5.7]{kis04} shows that if $r_{\pi', \iota}|_{G_F}$ is automorphic, then so is $r_{\pi, \iota}|_{G_F}$ (and then $\mathrm{BC}_{F / \bQ}(\pi)$ exists). If $n > 2$, then we can instead invoke \cite[Theorem 2.1]{New21b}. (Recall that if $n > 2$ then we are assuming the truth of Theorem \ref{thm_existence_of_base_change_with_supercuspidal_local_components} in the case $n = 2$, so the cited result does apply.) 
\end{proof}
\begin{prop}
Let $\pi$ be a regular algebraic, cuspidal automorphic representation of $\GL_2(\bA_\bQ)$ without CM, and such that $2, 3 \not\in sc(\pi)$. Then $\mathrm{BC}_{F / \bQ}(\Sym^{n-1} \pi)$ exists.
\end{prop}
\begin{proof}
The proof is identical to the proof of \cite[Proposition 3.11]{New21b}, using Proposition \ref{prop_putting_in_a_supercuspidal_prime} and Proposition \ref{prop_taking_out_a_supercuspidal_prime}.
\end{proof}
We now finish the proof of Theorem \ref{thm_existence_of_base_change_with_supercuspidal_local_components}.
\begin{proof}[Proof of Theorem \ref{thm_existence_of_base_change_with_supercuspidal_local_components}]
The proof is again quite similar to the proof of \cite[Theorem 3.1]{New21b}. We first treat the case where $\pi$ has weight 2 and $2 \not\in sc(\pi)$. In this case we can repeat the argument of \emph{loc. cit.}, replacing the use of \cite[Theorem 2.1]{New21b} by \cite[Theorem 3.5.7]{kis04} in the case $ n  = 2$, in order to conclude that $\mathrm{BC}_{F / \bQ}(\Sym^{n-1} \pi)$ exists. 

The next case is when $\pi$ has weight 2 and $2 \in sc(\pi)$. Here we need to be slightly more careful since Lemma \ref{lem_large_image_over_F} does not apply when $p = 2$. Choose a prime $t > \max(4N(N-1), N(\pi), [ F : \bQ])$ such that $t \equiv 1 \text{ mod }4$ and there is an isomorphism $\iota_t : \overline{\bQ}_t \to \bC$ such that $\operatorname{Proj} \overline{r}_{\pi, \iota}(G_\bQ)$ is conjugate either to $\PSL_2(\bF_t)$ or $\PGL_2(\bF_t)$. Choose a prime $q > N(\pi)$ satisfying the following conditions:
\begin{itemize}
\item $q \equiv -1 \text{ mod } t$, $q \equiv 1 \text{ mod }8$, $q \equiv 1 \text{ mod }l$ for every prime $l < t$.
\item $\overline{r}_{\pi, \iota_t}(\Frob_q)$ is in the conjugacy class of complex conjugation.
\end{itemize} 
We can then find another regular algebraic, cuspidal automorphic representation $\pi'$ of $\GL_2(\bA_\bQ)$ with the following properties:
\begin{itemize}
\item $\overline{r}_{\pi, \iota_t} \cong \overline{r}_{\pi', \iota_t}$.
\item For each prime $l \neq q$, $r_{\pi, \iota_t}|_{G_{\bQ_l}} \sim r_{\pi', \iota_t}|_{G_{\bQ_l}}$.
\item There is an isomorphism $\rec_{\bQ_q}(\pi'_q) \cong \Ind_{W_{\bQ_{q^2}}}^{W_{\bQ_q}} \chi$, where $\chi : W_{\bQ_{q^2}} \to \bC^\times$ is a character such that $\chi|_{I_{\bQ_q}}$ has order $t$. 
\end{itemize} 
Then Lemma \ref{lem_large_image_over_F} and \cite[Theorem 4.2.1]{BLGGT} together show that $\mathrm{BC}_{F / \bQ}(\Sym^{n-1} \pi)$ exists if and only if $\mathrm{BC}_{F / \bQ}(\Sym^{n-1} \pi')$ does. 

Let $\iota : \overline{\bQ}_2 \to \bC$ be an isomorphism. Then \cite[Lemma 6.3]{MR2551763} implies that the image of $\operatorname{Proj} \overline{r}_{\pi', \iota}$ is conjugate to $\PGL_2(\bF_{2^a})$ for some $2^a > 4$. Moreover, this image contains an element of order $t$. We deduce that $2^a > 2n-1$ (this is one of the hypotheses of \cite[Theorem 2.1]{New21b}, which we will be applying). By \cite[Lemma 3.4]{New21b}, we can find another regular algebraic, cuspidal automorphic representation $\pi''$ of $\GL_2(\bA_\bQ)$ such that the following conditions are satisfied:
\begin{itemize}
\item There is an isomorphism $\overline{r}_{\pi', \iota} \cong \overline{r}_{\pi'', \iota}$.
\item $\pi''$ has weight 2 and is not $\iota$-ordinary. The representation $\pi''_2$ is not supercuspidal.
\item For each prime $l$, $\pi''_l$ is a twist of the Steinberg representation if and only if $\pi'_l$ is. 
\end{itemize}
We observe that $\operatorname{Proj} \overline{r}_{\pi'', \iota}(G_F) = \operatorname{Proj} \overline{r}_{\pi'', \iota}(G_\bQ)$. Indeed, let $\widetilde{F}$ denote the Galois closure of $F / \bQ$, and let $K / \bQ$ denote the extension cut out by $\operatorname{Proj} \overline{r}_{\pi'', \iota}$. It suffices to check that $K \cap \widetilde{F} = \bQ$. However, $\Gal(K / \bQ)$ is a simple group of order divisible by $t$, while $\Gal(K \cap \widetilde{F} / \bQ)$ has order prime to $t$ (because of our assumption $t > [F : \bQ]$). 

If $n = 2$, then we have already show that $\mathrm{BC}_{F / \bQ}(\pi'')$ exists. Now \cite[Theorem 3.3.5]{MR2551765} implies that $r_{\pi', \iota}|_{G_F}$ is automorphic, hence that $\mathrm{BC}_{F / \bQ}(\pi)$ exists. If $n > 2$, then we apply \cite[Theorem 2.1]{New21b} to conclude that $\mathrm{BC}_{F / \bQ}(\Sym^{n-1} \pi)$ exists.

It remains finally to treat the case that $\pi$ has weight $k > 2$. The argument of the final paragraph of the proof of \cite[Theorem 3.1]{New21b} now applies to reduce this case to the one already treated. 
\end{proof}
%

\input{basechange.bbl}

\end{document}

%% file: basechange.bbl
\renewcommand{\MR}[1]{}